\RequirePackage[english]{babel}
\documentclass[a4paper,twoside,11pt]{amsart}
\usepackage[latin1]{inputenc}
\usepackage{amsmath}
\usepackage{amssymb}
\usepackage{amsfonts}
\usepackage{mathrsfs}
\usepackage{fixltx2e}[2005/12/01]
\usepackage{hyperref}

\usepackage{enumerate}
\renewcommand{\theenumi}{{\upshape{(\roman{enumi})}}}

\ifpdf
  \usepackage[matrix,arrow,cmtip,frame,graph,curve,2cell]{xy}
\else
  \usepackage[matrix,arrow,cmtip,frame,graph,curve,2cell,dvips]{xy}
\fi
\SelectTips{cm}{}
\newdir{(}{{}*!/-5pt/@^{(}}
\newdir{(x}{{}*!/-5pt/@_{(}}
\newdir{+}{{}*!/-9pt/{}}
\newdir{>+}{@{>}*!/-9pt/{}}
\entrymodifiers={+!!<0pt,\fontdimen22\textfont2>}
\UseAllTwocells

\usepackage[nosubsections]{mytheorems2}

\newtheoremstyle{dtheoremnopar}{3 mm}{1 mm}{\itshape}{}{\bfseries}{.}{ }
{\thmname{#1}\thmnumber{ #2}\thmnote{ \mdseries(#3)\bfseries}}

\theoremstyle{dtheoremnopar}
\newcounter{theoremx}
\newtheorem{theoremalpha}[theoremx]{Theorem}

\newcommand{\tref}[1]{\ref{#1}}

\newcommand{\pref}[1]{\eqref{#1}}
\newcommand\Z{\mathbb{Z}}
\newcommand\inj{\hookrightarrow}
\newcommand\surj{\twoheadrightarrow}
\newcommand\iso{\cong} 
\newcommand\map[3]{#1\colon #2\rightarrow #3}
\newcommand\injmap[3]{#1\colon #2\hookrightarrow #3}
\newcommand\id[1]{\mathrm{id}_{#1}} 
\newcommand\SG[1]{{\mathfrak{S}_{#1}}}   
\DeclareMathOperator{\Hom}{Hom}
\newcommand\sep{\mathrm{sep}}
\newcommand\sA{\mathcal{A}}
\newcommand\sF{\mathcal{F}}
\newcommand\sG{\mathcal{G}}

\newcommand\sO{\mathcal{O}}
\renewcommand\P{\mathbb{P}}           
\newcommand\A[1]{\mathbb{A}^{#1}}    
\newcommand\Spec{\mathrm{Spec}}
\newcommand\red{\mathrm{red}}   
\newcommand{\etale}{\'{e}tale}
\newcommand{\Etale}{\'{E}tale}
\newcommand\Hilb{\mathrm{Hilb}}
\newcommand\Mod{\mathbf{Mod}} 
\newcommand\Sch{\mathbf{Sch}}
\newcommand\QCoh{\mathbf{QCoh}} 

\newcommand{\stG}{\mathscr{G}} 
\newcommand{\SEC}{\mathrm{SEC}}
\newcommand{\ET}{\text{\textup{\'ET}}}
\newcommand{\op}{\mathrm{op}}

\newcommand\HilbSt{\mathscr{H}}

\newcommand\etalification{\'etalification}
\newcommand{\devissage}{d\'evissage}

\newcommand{\equalizer}[2]{\xymatrix@1@M=0mm@C=7mm{#1%
\ar@<.5ex>@{+->+}[r] \ar@<-.5ex>@{+->+}[r] & #2}}
\newcommand{\tLiset}{\operatorname{Lis-\acute{e}t}} 
\newcommand{\cart}{\mathrm{cart}} 
\newcommand{\fp}{\mathrm{fp}} 
\newcommand{\metale}{\mathrm{\acute{e}t}} 
\newcommand{\smooth}{\mathrm{sm}} 
\newcommand{\repr}{\mathrm{repr}} 
\newcommand{\qff}{\mathrm{qff}} 
\newcommand{\mopen}{\mathrm{open}} 

\newcommand\isoarrow{\overset{\iso}{\longrightarrow}}
\newcommand\isomap[3]{#1\,:\, #2\isoarrow #3}

\newcommand{\can}{\mathrm{can}} 

\newcommand{\catC}{\mathbf{C}}
\newcommand{\catD}{\mathbf{D}}
\newcommand{\catE}{\mathbf{E}}
\newcommand{\catF}{\mathbf{F}}
\newcommand{\catHom}{\mathbf{Hom}}
\DeclareMathOperator{\ob}{ob} 

\newcommand{\Cat}{\mathbf{Cat}}
\newcommand{\Grpd}{\mathbf{Grpd}}
\newcommand{\Stack}{\mathbf{Stack}}

\newcommand{\tfHom}{\mathbf{Hom}}
\newcommand{\tfQCoh}{\mathbf{QCoh}}
\newcommand{\tfAff}{\mathbf{Aff}}
\newcommand{\tfQAff}{\mathbf{QAff}}
\newcommand{\tfQProj}{\mathbf{QProj}}
\newcommand{\tfStackRep}{\mathbf{Stack}_{\repr}}
\newcommand{\tfF}{\mathbf{F}}

\newcommand\twomap[3]{#1\colon #2\Rightarrow #3}

\begin{document}

\title{\'Etale d\'evissage, descent and pushouts of stacks}
\author{David Rydh}
\address{KTH Royal Institute of Technology, Department of Mathematics,
SE\nobreakdash-100\ 44\ Stockholm, Sweden}
\thanks{Supported by the Swedish Research Council 2008-7143.}
\email{dary@math.kth.se}
\date{2011-01-10}


\begin{abstract}
We show that the pushout of an \etale{} morphism and an open immersion exists
in the category of algebraic stacks and show that such pushouts behave
similarly to the gluing of two open substacks. For example, quasi-coherent
sheaves on the pushout can be described by a simple gluing procedure. We then
outline a powerful \devissage{} method for representable \etale{} morphisms
using such pushouts. We also give a variant of the \devissage{} method
for representable quasi-finite flat morphisms.
%
\end{abstract}

\maketitle


\setcounter{secnumdepth}{0}
\begin{section}{Introduction}
Let $X$ be a scheme and let $X=U\cup V$ be an open cover of $X$. It is
well-known that:
\begin{enumerate}
\item\label{I:open-po:first}\label{I:open-po:descent}
Many objects over $X$ (such as quasi-coherent sheaves) correspond to
objects over $U$ and $V$ with a gluing datum over $U\cap V$. No cocycle
condition is needed as there are no non-trivial triple intersections.
\item The scheme $X$ is the pushout of the open immersions $U\cap V\to U$ and
$U\cap V\to V$.
\item\label{I:open-po:last}
Given two open immersions $W\subseteq U$ and $W\subseteq V$ of schemes
we can glue these to a scheme $X=U\cup_W V$. The scheme $X$ is the pushout
of $W\subseteq U$ and $W\subseteq V$ and we recover $W$ as the intersection
of $U$ and $V$.
\end{enumerate}
In \ref{I:open-po:first}--\ref{I:open-po:last}, we can also replace ``scheme''
by ``algebraic space'' or
``algebraic stack''. The purpose of this paper is to show that in the category
of algebraic spaces or algebraic stacks we can further extend these results,
taking one open immersion and one \emph{\etale{}} morphism instead of two open
immersions. We also outline a powerful \devissage{} method for \etale{}
morphisms
based upon these results as well as an extension to quasi-finite flat morphisms.

The simplest open coverings are of the type $X=U\cup V$ discussed above and
every open covering is a composition of such basic coverings. The \devissage{}
results explain how every \etale{} (resp.\ quasi-finite flat) morphism is built
up from \etale{} neighborhoods and finite \etale{} (resp.\ finite flat)
coverings.

To be able to state our results we need to make ``objects''
in~\ref{I:open-po:descent} above more
precise. Usually this is done in the language of fibered categories and
stacks. However as we need the base category to be the $2$-category $\Stack$ it
is more convenient to replace fibered categories and stacks with
\emph{$2$-functors} and
\emph{$2$-sheaves} as introduced by
R.\ Street~\cite{street_two-dim-sheaf-theory,street_char-bicat-stacks}. The
literature on $2$-sheaves is surprisingly meager and scattered. To not burden
this, essentially geometrical, paper with a long categorical treatment
of $2$-sheaves, we have chosen to give a short comprehensible
introduction in Appendix~\ref{A:2-sheaves} that exactly covers what we need.
Two examples of $2$-sheaves to keep in mind are the $2$-sheaf of quasi-coherent
sheaves $\map{\tfQCoh(-)}{\Stack^{\op}}{\Cat}$ and the $2$-sheaf
$\map{\tfHom(-,Y)}{\Stack^{\op}}{\Grpd}$ for a given algebraic stack $Y$,
cf.\ Appendix~\ref{A:ex-2-sheaves}.

Let $X$ be an algebraic stack and let $Z\subseteq |X|$ be a closed subset. An
\emph{\etale{} neighborhood} of $Z$ is an \etale{} morphism $\map{f}{X'}{X}$
such that the restriction
$\map{f|_{Z_\red}}{f^{-1}(Z_\red)}{Z_\red}$ is an isomorphism.
If $X=U\cup V$ is a union of two
open substacks, then $V\to X$ is an \etale{} neighborhood of $X\setminus U$.
Note that we do not require that $f$ is separated, nor representable, but
\etale{} signifies that $f$ is at least represented by Deligne--Mumford
stacks.
We can now state the main theorems of this paper, generalizing
\ref{I:open-po:first}--\ref{I:open-po:last} in
the beginning of the introduction.

\begin{theoremalpha}[Descent]\label{T:ETNBHD-DESCENT}
Let $X$ be an algebraic stack and let $U\subseteq X$ be an open substack. Let
$\map{f}{X'}{X}$ be an \etale{} neighborhood of $X\setminus U$ and let
$U'=f^{-1}(U)$. Let $\map{\tfF}{\Stack_{\metale/X}^{\op}}{\Cat}$ be a $2$-sheaf
in the
\etale{} topology, cf.\ Appendix~\ref{A:2-sheaves}. Then the natural functor
$$\map{(|_U,f^*)}{\tfF(X)}{\tfF(U)\times_{\tfF(U')} \tfF(X')}$$
is an equivalence of categories.
\end{theoremalpha}

\begin{theoremalpha}[\Etale{} neighborhoods are pushouts]\label{T:ETNBHD-IS-PUSHOUT}
Let
$$\xymatrix{U'\ar[r]^{j'}\ar[d]_{f|_U} & X'\ar[d]^f\\
U\ar[r]_j & X\ar@{}[ul]|\square}$$
be a cartesian diagram of algebraic stacks such that $\map{j}{U}{X}$ is an open
immersion and such that $\map{f}{X'}{X}$ is an \etale{} neighborhood of
$X\setminus U$. Then $X$ is the pushout of $f|_U$
and $j'$ in the category of algebraic stacks, that is, the cartesian square
is also co-cartesian.
\end{theoremalpha}

\begin{theoremalpha}[Existence of pushouts]\label{T:EXIST-PUSHOUTS}
Let $X'$ be an algebraic stack, let $\map{j'}{U'}{X'}$ be an open immersion
and let $\map{f_U}{U'}{U}$ be an \etale{} morphism. Then the pushout $X$ of
$j'$ and $f_U$ exists in the category of algebraic stacks. The resulting
co-cartesian diagram
$$\xymatrix{U'\ar[r]^{j'}\ar[d]_{f_U} & X'\ar[d]^f\\
U\ar[r]_{j} & X}$$
is also cartesian, $j$ is an open immersion and $f$ is an \etale{} neighborhood
of $X\setminus U$. Furthermore,
\begin{enumerate}
\item The formation of the pushout commutes with arbitrary base change.
\item If $f_U$ is representable, then so is $f$.
\item If $j'$ is quasi-compact then so is $j$.
If in addition $X'$ and $U$ are quasi-separated then so is $X$.
\item If $f_U$ is \emph{representable} and $X'$ and $U$ have
separated diagonals, then the diagonal of $X$ is separated.
\item If $X'$ and $U$ are algebraic spaces, then so is $X$.
\end{enumerate}
(also see Proposition~\pref{P:etnbhd-properties} for further properties)
\end{theoremalpha}

For the applications in mind, e.g., the \devissage{} method, it is useful to
have
Theorem~\tref{T:EXIST-PUSHOUTS} for \etale{} morphisms $f_U$ that are not
representable and in this case $X$ need not have separated diagonal, cf.\
Examples~\pref{E:bad-etnbhds}. It is thus natural to treat algebraic stacks
with non-separated diagonals. On the other hand, the queasy reader is
encouraged to assume that all algebraic stacks are at least quasi-separated,
i.e., have
quasi-compact and quasi-separated diagonals.

We will now state the \etale{} \devissage{} theorem.
Let $S$ be an algebraic stack. We let $\Stack_{\fp,\metale/S}$ denote the
$2$-category of \etale{} and finitely presented morphisms $X\to S$ and let
$\Stack_{\repr,\sep,\fp,\metale/S}$ denote the subcategory of morphisms that
are representable and separated. The second category is equivalent to a
$1$-category.

\begin{theoremalpha}[D\'evissage]\label{T:DEVISSAGE}
Let $S$ be a quasi-compact and quasi-separated algebraic stack and let $\catE$
be either $\Stack_{\fp,\metale/S}$ or $\Stack_{\repr,\sep,\fp,\metale/S}$. Let
$\catD\subseteq \catE$ be a full subcategory such that
{\renewcommand{\theenumi}{{\upshape{(D\arabic{enumi})}}}
\begin{enumerate}
\item if $X\in \catD$ and $(X'\to X)\in \catE$ then $X'\in \catD$,\label{TI:etdev:first}
\item if $X'\in \catD$ and $X'\to X$ is finite, surjective and \etale{},
  then $X\in \catD$, and
\item if $\map{j}{U}{X}$ and $\map{f}{X'}{X}$ are morphisms in $\catE$ such that
  $j$ is an open immersion and $f$ is an \etale{} neighborhood of $X\setminus
  U$, then $X\in \catD$ if $U,X'\in \catD$.\label{TI:etdev:last}
\end{enumerate}}
Then if $(X'\to X)\in \catE$ is representable and surjective and $X'\in
\catD$, we have that $X\in \catD$. In particular, if there exists a
representable and surjective morphism $X\to S$ in $\catE$ with $X\in \catD$
then $\catD=\catE$.
\end{theoremalpha}  

Theorem~\tref{T:DEVISSAGE} is generalized to quasi-finite flat morphisms
in Section~\ref{S:qff-devissage}.
Let us explain how Theorem~\tref{T:DEVISSAGE} usually is applied. Suppose that
we want to prove a statement $P(S)$ for an algebraic stack $S$ and that we know
that the corresponding statement $P(S')$ is true for some $S'$ where $S'\to S$
is representable, \etale{} and surjective. A typical situation is when $S$ is a
Deligne--Mumford stack and $S'\to S$ is a presentation. We let $\catD$ be the
subcategory of $\catE=\Stack_{\fp,\metale/S}$ of stacks $X\to S$ such that
$P(X)$ holds. It is then enough to verify
conditions~\ref{TI:etdev:first}--\ref{TI:etdev:last} for $\catD$ to deduce
that $P(S)$ holds. If $S'\to S$ is also separated, then we can work in the
smaller category $\Stack_{\repr,\sep,\fp,\metale/S}$ but if we do not assume
that $S'\to S$ is separated we have to include non-representable morphisms
even though $S'\to S$ is representable.

For \emph{algebraic spaces},
Theorems~\tref{T:ETNBHD-DESCENT}--\tref{T:EXIST-PUSHOUTS} are almost folklore.
Parts of
them or other closely related results appear in~\cite[\S5.7]{raynaud-gruson},\
\cite[\S4]{ferrand-raynaud_fibres-formelles},\ 
\cite[Thm.~2.6]{artin_alg_formal_moduli_II},
\cite[\S6.2]{raynaud-bosch-lutkebohmert_Neron-models},\ 
\cite{beauville-laszlo_lemme-de-descente},\ 
\cite{moret-bailly_un-probleme-de-descente}
and~\cite[\S3.1]{conrad-lieblich-olsson_Nagata}. The first aim of this paper is
to state and prove Theorems~\tref{T:ETNBHD-DESCENT}--\tref{T:EXIST-PUSHOUTS}
for \emph{algebraic stacks}, a highly
non-trivial task compared to the case with algebraic spaces. The second aim is
Theorem~\tref{T:DEVISSAGE} that explains and generalizes the \devissage{} method
that is implicit in~\cite[\S5.7]{raynaud-gruson}.

The \devissage{} method can be used to prove certain existence results that can
be
shown \etale{}-locally. This includes Raynaud--Gruson's flatification by
blow-ups~\cite{raynaud-gruson}, tame \etalification{} by
stacky blow-ups and compactifications of tame Deligne--Mumford
stacks~\cite{rydh_compactification-tame-stacks} and the
existence of absolute noetherian approximation of
stacks~\cite{rydh_noetherian-approx}. We also expect the \devissage{} method
to be useful in applications of a completely different flavor.

\Etale{} neighborhoods are frequently used in $K$-theory, motives and
$\A{1}$-homotopy theory. In this context they are known as \emph{elementary
  distinguished squares} or \emph{upper distinguished
  squares}, cf.\ \cite{morel-voevodsky_A1-homotopy-theory,voevodsky_homotopy-cd-topologies,voevodsky_homotopy-nis-cdh} and Remark~\pref{R:Nisnevich}.

\begin{subsection}{Outline}
In Section~\ref{S:etnbhd-descent} we prove Theorems~\tref{T:ETNBHD-DESCENT}
and~\tref{T:ETNBHD-IS-PUSHOUT}. In Section~\ref{S:etnbhds} we describe some
general properties of \etale{} neighborhoods and in
Section~\ref{S:etnbhd-pushouts} we give a proof of
Theorem~\tref{T:EXIST-PUSHOUTS}.
In Section~\ref{S:constructible} we show that every constructible sheaf is
locally constant on the stratification induced by an open filtration.
Equivalently, a representable \etale{} morphisms of finite presentation becomes
finite \etale{} after passing to such a stratification.
In Section~\ref{S:etale-devissage} we prove Theorem~\tref{T:DEVISSAGE} and
in Section~\ref{S:qff-devissage} we prove a more general \devissage{} result for
quasi-finite and flat morphisms. In Section~\ref{S:qf-stacks} we show that
every stack with quasi-finite diagonal has a quasi-finite presentation and
give an \etale{}-local structure theorem for such stacks.

In Appendix~\ref{A:separation} we state our conventions for algebraic stacks
and give some technical results on separation axioms for stacks.
In Appendix~\ref{A:alg-points} we show that points on quasi-separated stacks
are algebraic.
In Appendix~\ref{A:lemmas} we give two lemmas for algebraic spaces and
in Appendices~\ref{A:2-sheaves} and~\ref{A:ex-2-sheaves} we give a short
introduction to $2$-sheaves on stacks.

A morphism of stacks $\map{f}{X}{Y}$ is \emph{\etale} if and only if $f$ is
locally of finite presentation, flat and has \etale{} diagonal, cf.\
\cite[App.~B]{rydh_embeddings-of-unramified}.
\end{subsection}

\begin{subsection}{Acknowledgments}
I would like to thank L.\ Moret-Bailly, M.\ Olsson, M.\ Shulman, R.\ Skjelnes
and M.\ Temkin for
useful comments and discussions.
\end{subsection}

\end{section}
\setcounter{secnumdepth}{3}

\setcounter{tocdepth}{1} 
\tableofcontents


\begin{section}{Descent for \etale{} neighborhoods}\label{S:etnbhd-descent}
In this section we prove Theorems~\tref{T:ETNBHD-DESCENT}
and~\tref{T:ETNBHD-IS-PUSHOUT}.
Recall that if $X$ is an algebraic stack and $X=U_1\cup U_2$ is an open
covering, so that $U_2\to X$ is an open neighborhood of $X\setminus U_1$,
then a quasi-coherent sheaf on $X$ can be described as a pair of quasi-coherent
sheaves
$\sF_1\in\QCoh(U_1)$ and $\sF_2\in\QCoh(U_2)$ together with an isomorphism
$\sF_1|_{U_1\cap U_2}\to \sF_2|_{U_1\cap U_2}$.

The following notation will be fixed throughout this section.

\begin{notation}\label{N:etnbhd-functor}
Let $X$ be an algebraic stack and let $Z\inj X$ be a closed substack.  Let
$\map{f}{X'}{X}$ be an \etale{} neighborhood of $|Z|$, let $U=X\setminus Z$ and
let $U'=X'\setminus Z=f^{-1}(U)$. Let
$\map{\tfF}{\Stack_{\metale/X}^{\op}}{\Cat}$ be
a $2$-presheaf, i.e., a (pseudo) $2$-functor (cf.\ Appendix~\ref{A:2-sheaves}).
We have pull-back functors
\begin{align*}
(f|_U)^*&\colon \tfF(U)\rightarrow \tfF(U'), \\
|_{U'}&\colon \tfF(X')\rightarrow \tfF(U'),
\end{align*}
we can form the $2$-fiber product
$$\tfF(U)\times_{\tfF(U')} \tfF(X'),$$
and there is an induced functor
$$\map{(|_U,f^*)}{\tfF(X)}{\tfF(U)\times_{\tfF(U')} \tfF(X')}$$
that is unique up to unique natural isomorphism.
\end{notation}

Under the assumption that $\tfF$ is a $2$-sheaf in the \etale{} topology, we
will show that the functor $(|_U,f^*)$ is an equivalence of categories. This is
Theorem~\tref{T:ETNBHD-DESCENT}.
Examples of $2$-sheaves include $\tfQCoh$ and $\tfHom(-,Y)$,
cf.\ Appendix~\ref{A:ex-2-sheaves}.

\begin{example}
Let $\tfF=\tfQCoh$ be the $2$-functor of quasi-coherent sheaves. Then
$$\QCoh(U)\times_{\QCoh(U')} \QCoh(X')$$
is the category of quasi-coherent sheaves on $U$ with a specified extension to
$X'$. More formally, the objects are triples $(\sF_U,\theta,\sF')$ where
$\sF_U\in \QCoh(U)$, $\sF'\in \QCoh(X')$ and
$\map{\theta}{(f|_U)^*{\sF_U}}{\sF'|_{U'}}$ is an isomorphism. The morphisms
are pairs $\map{(\varphi_U,\varphi')}{(\sF_U,\theta,\sF')}{(\sG_U,\psi,\sG')}$
where $\map{\varphi_U}{\sF_U}{\sG_U}$ and $\map{\varphi'}{\sF'}{\sG'}$ are
homomorphisms such that $\varphi'|_{U'}\circ
\theta=\psi\circ(f|_U)^*\varphi_U$.
\end{example}

\begin{proof}[Proof of Theorem~\tref{T:ETNBHD-DESCENT}]
Let $\map{\pi_1,\pi_2}{X'\times_X X'}{X'}$ be the two projections, let
$\map{\Delta_{X'/X}}{X'}{X'\times_X X'}$ be the diagonal and let
$\map{\pi}{X'\times_X X'}{X}$ denote the structure morphism. The key
observation is that the assumptions on $f$ imply that
$$\map{h=(j''\amalg
\Delta_{X'/X})}{(U'\times_U U')\amalg X'}{X'\times_X X'}$$
is \etale{}, representable and surjective. Here $\map{j''}{U'\times_U
U'}{X'\times_X X'}$ denotes the canonical open immersion.

We first show that the functor $(|_U,f^*)$ is fully faithful. Let $\sF,\sG\in
\tfF(X)$ be two objects. Replacing $X'$ with $X'\amalg U$ we can assume that
$f$ is surjective. As $f$ is a morphism of descent, the sequence
$$\xymatrix{
\Hom(\sF,\sG)\ar[r]
 & \Hom(\sF',\sG')\ar@<.5ex>[r]^-{\pi_1^*}\ar@<-.5ex>[r]_-{\pi_2^*}
 & \Hom(\sF'',\sG'')
}$$
is exact where $\sF'=f^*\sF$, $\sF''=\pi^*\sF$ etc. As $h$ is \etale{} and
surjective, the map
$$\xymatrix{
\Hom(\sF'',\sG'')\ar[r]^-{h^*} & \Hom(h^*\sF'',h^*\sG'')
}$$
is injective. Given compatible morphisms $\map{\varphi_U}{\sF|_U}{\sG|_U}$ and
$\map{\varphi'}{\sF'}{\sG'}$, we have that
$h^*\pi_1^*\varphi'=h^*\pi_2^*\varphi'$ since both morphisms coincide with
$(\pi|_U)^*\varphi_U$ on the first component $U'\times_U U'$ and with
$\varphi'$ on the second component $X'$. Thus, by descent, there is a unique
morphism $\map{\varphi}{\sF}{\sG}$ such that $\varphi'=f^*\varphi$ and
$\varphi_U=\varphi|_U$.

Next, we show that the functor is essentially surjective. Let
$\sF_U\in\tfF(U)$ and $\sF'\in\tfF(X')$ be objects together with an
isomorphism $\map{\theta}{(f|_U)^*{\sF_U}}{\sF'|_{U'}}$. The isomorphism
$\theta$ provides $\sF'|_{U'}$ with a descent datum, i.e., an isomorphism
$\map{\psi_{U'}}{(\pi_1^*\sF')|_{U'\times_U U'}}{(\pi_2^*\sF')|_{U'\times_U
U'}}$ satisfying the cocycle condition over $U'\times_U U'\times_U U'$. The
cocycle condition implies that $(\Delta_{U'/U})^*\psi_{U'}$ is the identity on
$\sF'|_{U'}$.

As we have seen, the functor
$\map{h^*=(j''^*,(\Delta_{X'/X})^*)}{\tfF(X'\times_X X')}{\tfF(U'\times_U
U')\times_{\tfF(U')} \tfF(X')}$ is fully faithful so the isomorphism
${(\psi_{U'},\id{X'})}$ descends to a unique isomorphism
$\map{\psi}{\pi_1^*\sF'}{\pi_2^*\sF'}$ such that $\psi|_{U'\times_U
U'}=\psi_{U'}$ and $(\Delta_{X'/X})^*\psi=\id{X'}$. Finally, $\psi$ satisfies
the cocycle condition since $X'\times_X X'\times_X X'$ has an
\etale{} cover consisting of the open substack $U'\times_U U'\times_U U'$ and
the diagonal $X'\to X'\times_X X'\times_X X'$. By effective descent, we obtain
an object $\sF\in\tfF(X)$ that restricts to $\sF_U$ and $\sF'$.
\end{proof}

\begin{remark}
Let $Z_1\to X$ and $Z_2\to X$ be morphisms of algebraic
stacks. Theorem~\tref{T:ETNBHD-DESCENT} applied to the $2$-sheaf
$\tfF=\tfHom_X(Z_1\times_X -,Z_2)$ shows that a morphism
$\map{\varphi_U}{Z_1|_U}{Z_2|_U}$ which extends to a morphism
$\map{\varphi'}{Z_1'}{Z_2'}$ descends to a morphism $\map{\varphi}{Z_1}{Z_2}$
that is unique up to unique $2$-isomorphism. It can also be shown that a stack
$Z_U$ over $U$ extending to a stack $Z'$ over $X'$ glues to a stack $Z$
over $X$ that is unique up to unique $2$-isomorphism, cf.\
Corollary~\pref{C:effective-etnbhd-descent-of-stacks}.

A natural way to formalize these two results is to let $\tfF$ be the ``fibered
$2$-category of stacks'' so that $\tfF(X)$ is the $2$-category of stacks over
$X$. The results are then equivalent to the statement that the $2$-functor
\[
\map{(|_U,f^*)}{\tfF(X)}{\tfF(U)\times_{\tfF(U')} \tfF(X')}
\]
is a
$2$-equivalence of $2$-categories.
The proof is straight-forward except that one has
to deal with unpleasant objects such as functors of tricategories or
fibered $2$-categories of stacks. The canonical descent datum in this setting
consists of a $1$-isomorphism over $(X'/X)^2=X'\times_X X'$ and a
$2$-isomorphism over $(X'/X)^3$ satisfying a cocycle condition over
$(X'/X)^4$. All these technical issues can be completely avoided using
Theorem~\tref{T:EXIST-PUSHOUTS} as is done in
Corollary~\pref{C:effective-etnbhd-descent-of-stacks}.
\end{remark}

\begin{proof}[Proof of Theorem~\tref{T:ETNBHD-IS-PUSHOUT}]
Let $W$ be an algebraic stack and let
$$\xymatrix{U'\ar[r]^{j'}\ar[d]_{f|_U}\drtwocell<\omit>{^\varphi} & X'\ar[d]^{g'}\\
U\ar[r]_{g_U} & W}$$
be $2$-commutative. We have to show that there is a morphism $\map{g}{X}{W}$
and a $2$-commutative diagram
\begin{equation}\label{E:pushout-universal-property}
\vcenter{\xymatrix{U'\ar[r]\ar[d]\drtwocell<\omit>{^\tau} & X'\ar[d]_f\ar@/^/[rdd]^{g'}
\ddrtwocell<\omit>{^<1>\eta'} \\
U\ar[r]_j\ar@/_/[rrd]_{g_U}\drrtwocell<\omit>{<-1>\;\;\;\eta_U}
 &  X\ar[rd]_g\\
&& W}}
\end{equation}
such that the pasting of the diagram is $\varphi$.

By Theorem~\pref{T:Hom-is-2-sheaf}, $\map{\tfF=\tfHom(-,W)}{\Stack^\op}{\Grpd}$
is a $2$-sheaf. By
Theorem~\tref{T:ETNBHD-DESCENT} the object
$$(g_U,\varphi,g')\in \tfF(U)\times_{\tfF(U')} \tfF(X')$$
descends to an object $g\in \tfF(X)=\tfHom(X,W)$ together with a $2$-morphism
$\twomap{\eta}{(g\circ j,\tau,g\circ f)}{(g_U,\varphi,g')}$, i.e., we
have two $2$-morphisms $\twomap{\eta_U}{g\circ j}{g_U}$ and
$\twomap{\eta'}{g\circ f}{g'}$ such that the pasting of
diagram~\eqref{E:pushout-universal-property} is $\varphi$.

Moreover, as $(|_U,f^{*})$ is fully faithful, any two solutions $(g,\eta)$ and
$(\widetilde{g},\widetilde{\eta})$ are uniquely $2$-isomorphic. Specifically,
there is a unique $2$-isomorphism $\twomap{\psi}{g}{\widetilde{g}}$ such that
%
$$\xymatrix{
U\ar[r]\rrlowertwocell<-12>_{g_U}{\;\;\widetilde{\eta}_U}
 & X\rtwocell<2.5>^{g}_{\widetilde{g}}{\psi} & W} \quad = \quad
\xymatrix{
U\ar[r]\rrlowertwocell<-9>_{g_U}{\;\;\eta_U}
 & X\ar[r]^{g} & W}$$
and such that the analogous identity involving $\eta'$ and $\widetilde{\eta}'$
holds.
\end{proof}

\begin{remark}
The special case of Theorem~\tref{T:ETNBHD-DESCENT} when $X'$ and $X$ are
spectra of DVRs can be found
in~\cite[6.2, C]{raynaud-bosch-lutkebohmert_Neron-models}.

Theorems~\tref{T:ETNBHD-DESCENT} and~\tref{T:ETNBHD-IS-PUSHOUT} immediately
generalize to \emph{flat and finitely presented} neighborhoods. Indeed, if
$X'\to X$ is flat and finitely presented and an isomorphism over a closed
substack $Z\inj X$, then $X'\to X$ is \etale{} in an open
neighborhood of $Z\inj X'$. On the other hand, the straight-forward
generalization of Theorem~\tref{T:EXIST-PUSHOUTS} to the flat and finitely
presented case does not hold.

A more interesting generalization is when $X'\to X$ is a flat and quasi-compact
neighborhood, e.g., the completion along a closed subscheme $Z\inj X$. The
case when $X'\to X$ is an affine and flat neighborhood is treated in
\cite[\S4]{ferrand-raynaud_fibres-formelles},\ 
\cite{moret-bailly_un-probleme-de-descente},\ 
\cite[Thm.~2.6]{artin_alg_formal_moduli_II} (for completion),\ 
\cite{beauville-laszlo_lemme-de-descente} (for completion along a hypersurface)
and~\cite[6.2, D]{raynaud-bosch-lutkebohmert_Neron-models} (for DVRs). The
flat analogue of Theorem~\tref{T:ETNBHD-DESCENT} is known for:
\begin{enumerate}
\item The $2$-sheaves $\tfQCoh$, $\tfAff$, $\tfQAff$, $\tfQProj$ of
  quasi-coherent sheaves, affine morphisms, quasi-affine morphisms
  and morphisms equipped with an ample line
  bundle, cf.\ \cite[Prop.~4.2]{ferrand-raynaud_fibres-formelles}
  and~\cite[Thm.~1.1]{moret-bailly_un-probleme-de-descente}.
\item The $2$-sheaf $\tfHom(-,Y)$ for an algebraic stack $Y$ with quasi-affine
  diagonal~\cite[Cor.~6.5.1]{moret-bailly_un-probleme-de-descente}.
\end{enumerate}
In particular, the flat variant of Theorem~\tref{T:ETNBHD-IS-PUSHOUT} holds
in the category of algebraic stacks with quasi-affine diagonal. The restriction
to quasi-affine diagonal is needed to ensure that the stacks are also stacks
in the fpqc-topology~\cite[Cor.~10.7]{laumon}.
\end{remark}

\end{section}


\begin{section}{\Etale{} neighborhoods}\label{S:etnbhds}
Let $X$ be an algebraic stack, let $Z\subseteq |X|$ be a closed subset and let
$U=X\setminus Z$. Let $\map{f}{X'}{X}$ be an \etale{} neighborhood of $Z$. In
this section we study how properties of $X$ and $f$ are related to properties
of $U$, $X'$ and $f|_U$, e.g., $f$ is representable if and only if $f|_U$ is
representable. We begin by showing that the notion of being an \etale{}
neighborhood of $Z$ is set-theoretic and does not depend on the choice of a
substack structure on $Z$.

\begin{lemma}\label{L:etnbhd-chars}
Let $\map{f}{X'}{X}$ be an \etale{} morphism of algebraic stacks and let
$Z\subseteq |X|$ be a closed subset. The following are equivalent
\begin{enumerate}
\item\label{LI:etnbhd:arb}
For every morphism $\map{g}{T}{X}$ such that $g(T)\subseteq Z$, the
projection $X'\times_X T\to T$ is an isomorphism.
\item\label{LI:etnbhd:red}
The projection $X'\times_X Z_\red\to Z_\red$ is an isomorphism, i.e.,
$f$ is an \etale{} neighborhood of $Z$.
\item\label{LI:etnbhd:point}
For every field $k$ and point $\map{x}{\Spec k}{X}$ in $Z$,
the fiber $X'_x\to \Spec k$ is an isomorphism.
\end{enumerate}
\end{lemma}
\begin{proof}
Clearly \ref{LI:etnbhd:arb}$\implies$\ref{LI:etnbhd:red}$\implies$\ref{LI:etnbhd:point}.
That \ref{LI:etnbhd:point}$\implies$\ref{LI:etnbhd:arb} follows
immediately from the following two facts.
%
A morphism which is locally of finite type and such that every fiber is
an isomorphism is a surjective monomorphism~\cite[Prop.~17.2.6]{egaIV}.
A surjective \etale{} monomorphism is an
isomorphism~\cite[Thm.~17.9.1]{egaIV}.
%
\end{proof}

\begin{xpar}[Inertia stacks]
Let $\map{f}{X}{Y}$ be a morphism of stacks. Then there is an induced morphism
of inertia stacks $\map{I_f}{I_X}{I_Y}$. The morphism $I_f$ is a composition
$I_{X}\to I_Y\times_Y X\to I_Y$ where the first morphism is a pull-back of
$\Delta_f$ and the second morphism is a pull-back of $f$. In particular, if $f$
is \etale{} (resp.\ an open immersion) then so is $I_f$.
\end{xpar}

\begin{xpar}
Given a cartesian diagram of stacks
$$\xymatrix{%
U'\ar[r]^{j'}\ar[d]_{f_U} & X'\ar[d]^f\\
U\ar[r]_j & X\ar@{}[ul]|\square}$$
we have the following cartesian diagram of stacks
$$\vcenter{\xymatrix{%
U'\ar[r]^{j'}\ar[d]_{\Delta_{f_U}} & X'\ar[d]^{\Delta_f}\\
U'\times_U U'\ar[r]_-{j'\times j'} & X'\times_X X'\ar@{}[ul]|\square}}
\quad\text{and}\quad
\vcenter{\xymatrix{%
I_{U'}\ar[r]^{I_{j'}}\ar[d]_{I_{f_U}} & I_{X'}\ar[d]^{I_f}\\
I_{U}\ar[r]_{I_{j}} & I_{X}.\ar@{}[ul]|\square}}$$
\end{xpar}

\begin{proposition}\label{P:etnbhd-properties}
Let $\map{j}{U}{X}$ be an open immersion of stacks and let $\map{f}{X'}{X}$ be
an \etale{} neighborhood of $X\setminus U$. Let $\map{j'}{U'}{X'}$ be the
pull-back of $j$ along $f$. Then
\begin{enumerate}
\item $\map{f\amalg j}{X'\amalg U}{X}$ is \etale{} and surjective.
\label{PI:etnbhd:etsurj}
\item $\Delta_f$ is an \etale{} neighborhood of
$X'\times_X X'\setminus U'\times_U U'$.
\label{PI:etnbhd:nbhd-diagonal}
\item $I_f$ is an \etale{} neighborhood of $I_X\setminus I_U$.
\label{PI:etnbhd:nbhd-inertia}
\item $I_{X'}\to I_{X}\times_X X'$ is an \etale{} neighborhood of
$I_{X}\times_X X'\setminus I_{U}\times_U U'$.
\label{PI:etnbhd:nbhd-inertia2}
\item If $j'$ is quasi-compact then so is $j$.\label{PI:etnbhd:j-qc}
\item If $f|_U$ is an open immersion (resp.\ a quasi-compact open
immersion, resp.\ an open and closed immersion, resp.\ an isomorphism,
resp.\ surjective) then so is
$f$.\label{PI:etnbhd:open/clopen/iso/surj/qcopen}
\item If $f|_U$ is representable (resp.\ representable and quasi-separated,
resp.\ representable and separated) then so
is~$f$.\label{PI:etnbhd:repr/repr+sep}
\item If $j'$ is quasi-compact and
$f|_U$ is quasi-compact (resp.\ quasi-separated, resp.\ of finite presentation)
then $f$ is quasi-compact (resp.\ quasi-separated, resp.\ of finite
presentation).\label{PI:etnbhd:qc/qs}
\item If $U$ and $X'$ are algebraic spaces, then so is $X$.
\label{PI:etnbhd:alg-space}
\item If $j'$ is quasi-compact and $U$ and $X'$ are
quasi-separated, then $X$ is quasi-separated.\label{PI:etnbhd:qsep}
\item If $f|_U$ is \emph{representable} and $U$ and $X'$ have
separated (resp.\ locally separated) diagonals, then $X$ has separated
(resp.\ locally separated) diagonal.\label{PI:etnbhd:sep/lsep-diag}
\end{enumerate}
\end{proposition}
\begin{proof}
\ref{PI:etnbhd:etsurj}--\ref{PI:etnbhd:nbhd-inertia2} are obvious.

\ref{PI:etnbhd:j-qc}: If $j'$ is quasi-compact then so is $j$ since the
pull-back of $j$ along the \etale{} surjective morphism $f\amalg j$ is
$j'\amalg\id{U}$.

\ref{PI:etnbhd:open/clopen/iso/surj/qcopen}:
If $f|_U$ is an open immersion, then $f$ is a monomorphism and hence an open
immersion. If in addition $f|_U$ is quasi-compact (resp.\ closed), then so is
$f$ since the pull-back of $f$ along the open covering $X'\amalg U\to X$ is
quasi-compact (resp.\ closed). If $f|_U$ is surjective, then so is $f$.

\ref{PI:etnbhd:repr/repr+sep}:
Apply~\ref{PI:etnbhd:open/clopen/iso/surj/qcopen} to the diagonal
$\Delta_{X'/X}$.

\ref{PI:etnbhd:qc/qs}:
Assume that $f|_U$ and $j'$ are quasi-compact. The pull-back of $f$ along
$f\amalg j$ is $\pi_2\amalg f|_U$ where $\map{\pi_2}{X'\times_X X'}{X'}$ is
the second projection. By assumption $f|_U$ is quasi-compact and $\pi_2$ is
quasi-compact as the composition $\map{\pi_2\circ (\Delta_{X'/X}\amalg
(j'\times j'))}{X'\amalg (U'\times_U U')}{X'}$ is quasi-compact. Thus, $f$
is quasi-compact.

Similarly, if $j'$ and the diagonal of $f|_U$ are quasi-compact we apply
the previous argument to the \etale{} neighborhood $\Delta_f$ of
$X'\times_X X'\setminus U'\times_U U'$ and conclude that $\Delta_f$ is
quasi-compact. In particular, if $f|_U$ is quasi-separated (i.e., if its
diagonal and the diagonal of the diagonal are quasi-compact) it follows that
$f$ is quasi-separated.

\ref{PI:etnbhd:alg-space}:
Assume that $U$ and $X'$ are algebraic spaces
so that $I_U=I_X\times_X U\to U$ and $I_{X'}\to X'$ are isomorphisms.  To show
that $X$ is an algebraic space it is thus enough to show that $I_X\times_X
X'\to X'$ is an isomorphism. By~\ref{PI:etnbhd:nbhd-inertia2} we have an
\etale{} neighborhood as described by the diagram
$$\xymatrix{%
U'=I_{U'}\ar[r]\ar[d] & X'=I_{X'}\ar[d]\\
U'=I_{U}\times_U U'\ar[r] & I_{X}\times_X X'}$$
and it follows that $I_X\times_X X'=X'$
by~\ref{PI:etnbhd:open/clopen/iso/surj/qcopen}.

\ref{PI:etnbhd:qsep}:
First note that $f|_U$ is quasi-separated
so that $f$ is quasi-separated
by~\ref{PI:etnbhd:qc/qs}. We have to prove that $\Delta_X$ is quasi-compact and
quasi-separated. This is an \etale{}-local question on $X\times X$ so it is
enough to show that the pull-backs of $\Delta_X$ along $j\times \id{X}$,
$\id{X}\times j$ and $f\times f$ are quasi-compact and quasi-separated. The
first pull-back is $\map{(\id{U},j)}{U}{U\times X}$ which is quasi-compact and
quasi-separated since $j$ is quasi-compact and quasi-separated and $U$ is
quasi-separated. The second pull-back is similar to the first one. The third
pull-back is $X'\times_X X'\to X'\times X'$. Since $f$
is quasi-separated,
$$\map{\Delta_f\amalg j'\times j'}{X'\amalg U'\times_U U'}{X'\times_X X'}$$
is quasi-compact, quasi-separated and surjective. It thus follows that
$X'\times_X X'\to X'\times X'$ is quasi-compact and quasi-separated from
Lemma~\pref{L:quasi-separatedness} since $\Delta_{X'}$, $\Delta_{U}$ and $j'$
are quasi-compact and quasi-separated.


\ref{PI:etnbhd:sep/lsep-diag}:
Assume that $f|_U$ is representable so that $f$
is representable by~\ref{PI:etnbhd:repr/repr+sep}. Further assume that
$\Delta_U$ and $\Delta_{X'}$ are separated (resp.\ locally separated) so that
the unit sections $U\to I_U$ and $X'\to I_{X'}$ are closed immersions
(resp.\ immersions). To see that $\Delta_X$ is separated (resp.\ locally
separated), we have to show that $X\to I_X$ is a closed immersion (resp.\ an
immersion).
%
As $X'\amalg U\to X$ is \etale{} and surjective and $U\to I_U=I_X\times_X U$ is
a closed immersion (resp.\ an immersion), it is enough to show that
$X'\to I_X\times_X X'$ is a closed immersion (resp.\ an immersion). As $f$ is
representable, we have that $I_X\times_X X'$ is the
union of two open subsets $I_U\times_U U'$ and $I_{X'}$. The restrictions of
$X'\to I_X\times_X X'$ to these open subsets are $U'\to I_U\times_U U'$ and
$X'\to I_{X'}$ which both are closed immersions (resp.\ immersions).
\end{proof}

\begin{examples}\label{E:bad-etnbhds}
We give some examples showing that $X$ and $f$ can be rather ``bad'' even if
$f|_U$, $X'$ and $U$ are ``nice''.
\begin{enumerate}
\item ($f|_U$ finite but $f$ not proper) Let $U\subset X$ be an open non-closed
subset, let $X'=U\amalg X$ and let $\map{f}{X'}{X}$ be the natural morphism.
Then $f$ is a non-proper \etale{} neighborhood of $X\setminus U$ and $f|_U$
is finite.
\item ($f|_U$ proper but $f$ not separated) Let $U\subset X$ be an open
non-closed subset and let $G\to X$ be the group scheme $G=X\amalg U\subset
X\times \Z/2\Z$. Let $X'=\mathrm{B}G=[X/G]$ so that
$U'=U\times \mathrm{B}(\Z/2\Z)$. Then
$\map{f}{X'}{X}$ is an \etale{} non-separated neighborhood of $X\setminus U$
and $f|_U$ is proper (a trivial \etale{} $\Z/2\Z$-gerbe).
\item ($f|_U$ proper, $U$ and $X'$ separated but $\Delta_X$ not separated)
Let $U\subset Y$ be an open non-closed subset. Let $G'=Y\times \Z/2\Z$ be
the constant group scheme, let $H=Y\amalg U\subset G'$ be the induced
subgroup and let $G=Y\amalg_U Y=G'/H$ so that $G$ is a non-separated
group scheme. Let $X=[Y/G]$ and $X'=[Y/G']$ where both group actions are
trivial. Then $\map{f}{X'}{X}$
is an \etale{} neighborhood of $X\setminus U$ such that $U$ and $X'$ are
separated and $f|_U$ is proper (a trivial \etale{} $\Z/2\Z$-gerbe).
\item ($f|_U$ proper, $U$ and $X'$ separated but $\Delta_X$ not locally
separated)
Let $p$ be a prime, let $G'=\mu_{p,\Z}$, let
$H=\mu_{p,\Z[1/p]}\cup_{\Spec \Z[1/p]} \Spec \Z\inj G'$ and
let $G=G'/H$ so that $G$ is not locally separated. Indeed, $G\to \Spec \Z$ is a
flat birational universal homeomorphism but not an isomorphism. Then let
$X=[\Spec(\Z)/G]$, $X'=[\Spec(\Z)/G']$ and $U=\Spec \Z[1/p]$ so that
$f|_U$ is a trivial $\mu_{p,\Z[1/p]}$-gerbe.
\item ($f|_U$ finite, $U$ and $X'$ separated schemes but $X$ not separated)
Let $U=\A{1}$ be the affine line, let $U'=U\amalg U$ and let $X'=\P^{1}\amalg
\P^{1}$. Then $X$ is a non-separated ``projective'' line.
\item ($f|_U$ finite, $U$ and $X'$ separated schemes but $X$ not locally
separated)
Let $U=\A{1}$ be the affine line, let $U'=U\amalg U$ and let $X'=\P^{1}\amalg_{\infty} \P^{1}$ be two secant lines. Then $X$ is a standard example of a
non-locally separated algebraic space.
\end{enumerate}
\end{examples}

\end{section}


\begin{section}{\Etale{} gluings of stacks}\label{S:etnbhd-pushouts}
In this section we will prove Theorem~\tref{T:EXIST-PUSHOUTS} on the existence
of the pushout of an open immersion $\map{j'}{U'}{X'}$ and an \etale{}
morphism $\map{f_U}{U'}{U}$.
If $U$, $U'$ and $X'$ are \emph{algebraic spaces}, then it is rather
straight-forward to construct the pushout $X$ of $j'$ and $f_U$. Indeed, by
Theorem~\tref{T:ETNBHD-IS-PUSHOUT} we know a priori that $X'\times_X X'$ has to
be the pushout of $j'$ and $\Delta_{U'/U}$ and by assumption both these maps
are open immersions so we can construct the algebraic space $R'=X'\times_X X'$
as this pushout. The universal property of the pushout gives a morphism $R'\to
X'\times X'$ and it can be shown that this is an \etale{} equivalence
relation. The space $X$ is then the quotient of this equivalence
relation. Similarly, if $U'$ and $X'$ are algebraic spaces and $U$ is an
algebraic stack, then we can construct $R'$ as above and equip $(R',X')$ with a
groupoid structure although it is slightly tedious to verify that this is
indeed a groupoid.

For arbitrary $X'$ this procedure is not so straight-forward as the groupoid
$\equalizer{R'}{X'}$ would be a groupoid in stacks (and even with
non-representable morphisms if $f_U$ is not representable!). The most natural
approach is to first define the quotient $X$ as a $2$-stack and then show that
$X$ is
equivalent to a $1$-stack. That $X$ is indeed a $1$-stack follows from the
fact that the stabilizer $R'\times_{X'\times X'} X'\to X'$ is representable.
A brief description of how $X$ can be constructed in this manner is given
in~\cite[2.5.1]{romagny_components-in-families}.
However, to avoid the language of $2$-stacks and groupoids in stacks we will
do an explicit, albeit somewhat less natural, construction.

\begin{proof}[Proof of Theorem~\tref{T:EXIST-PUSHOUTS}]
When $\map{f_U}{U'}{U}$ is a monomorphism (resp.\ representable,
resp.\ arbitrary) then the
diagonal $\Delta_{U'/U}$ is an isomorphism (resp.\ a monomorphism, resp.\
representable). We will assume that the theorem is true when $f_U$ is an
isomorphism (resp.\ a monomorphism, resp.\ representable) and show that the
theorem is true when $f_U$ is a monomorphism (resp.\ representable, resp.\
arbitrary). When $f_U$ is an isomorphism, then $X=X'$ is the pushout.

We can thus assume that the pushout $R'$ of $\Delta_{U'/U}$ and $j'$ exists
and fits into the bi-cartesian square
$$\xymatrix{U'\ar[r]^{j'}\ar[d]_{\Delta_{U'/U}} & X'\ar[d]^{\Delta}\\
U'\times_U U'\ar[r] & R'.}$$
For $k=1,2$, the morphisms $\map{j'\circ\pi_k}{U'\times_U U'}{X'}$ and
$\id{X'}$ induce a morphism $\map{q_k}{R'}{X'}$. When the existence of the
pushout $X$ has been settled, then $R'=X'\times_X X'$ and under this
identification $\Delta$ becomes the diagonal and $q_k$ the projection onto the
$k$\textsuperscript{th} factor.

Let $\map{p}{X'_1}{X'}$ be a smooth presentation and let
$\equalizer{X'_2}{X'_1}$ be the induced groupoid with quotient $X'$.
Let $\equalizer{U'_2}{U'_1}$
be the pull-back of the groupoid along $j'$.

Consider the following category $\catC$ fibered over $\Sch$ (we will eventually
show that this is the pushout $X$). An object of $\catC$ over $T\in\Sch$
consists of
\begin{enumerate}
\item an open subset $T_\circ \subseteq T$,
\item a groupoid $\equalizer{T'_2}{T'_1}$ over $T$,
\item a morphism $\map{g_\circ}{T_\circ}{U}$,
\item morphisms $\map{g'_1}{T'_1}{X'_1}$ and $\map{g'_2}{T'_2}{X'_2}$,
\end{enumerate}
such that
{\renewcommand{\theenumi}{{\upshape{(\alph{enumi})}}}
\begin{enumerate}
\item the diagrams
$$\vcenter{\xymatrix{T'_2\ar[r]^{g'_2}\ar[d]_s & X'_2\ar[d]_s\\
T'_1\ar[r]_{g'_1} & X'_1}}\quad\text{and}\quad
\vcenter{\xymatrix{T'_2\ar[r]^{g'_2}\ar[d]_t & X'_2\ar[d]_t\\
T'_1\ar[r]_{g'_1} & X'_1}}$$
are cartesian,
\item the inverse images $T'_{1\circ}\subseteq T'_1$ and
$T'_{2\circ}\subseteq T'_2$ of $T_\circ\subseteq T$ coincide with
$(g'_1)^{-1}(U'_1)$ and $(g'_2)^{-1}(U'_2)$,
\item the diagram
$$\xymatrix{T'_{1\circ}\ar[r]\ar[d] & U'_1\ar[d]\\ T_{\circ}\ar[r] & U}$$
is cartesian,
\item the stack quotient $T'=[\equalizer{T'_2}{T'_1}]$ is an \etale{}
neighborhood of $T\setminus T_\circ$ in $T$.
\end{enumerate}}

For an object $(T_\circ\subseteq T,T'_{\bullet}\to T,g_\circ,g'_{\bullet})$ we
thus obtain a cartesian diagram
$$\xymatrix{U'\ar[r]^{j'}\ar[d]_{f_U} & X'\\
U & T'_\circ\ar[r]\ar[d]\ar[ul] & T'\ar[d]\ar[ul]_{g'}\ar@{}[ull]|\square\\
  & T_\circ\ar[r]\ar[ul]_{g_\circ}\ar@{}[uul]|\square & T.\ar@{}[ul]|\square}$$
where the bottom-right square is bi-cartesian by
Theorem~\tref{T:ETNBHD-IS-PUSHOUT}.
%

A morphism in $\catC$ over $S\to T$ is a commutative diagram
$$\xymatrix{S'_2\ar@<.5ex>[r]\ar@<-.5ex>[r]\ar[d] & S'_1\ar[r]\ar[d]
 & S\ar[d] & S_\circ\ar[l]\ar[d]\dduppertwocell{<-2>}\\
T'_2\ar@<.5ex>[r]\ar@<-.5ex>[r]\ar[d] & T'_1\ar[r]\ar[d]
 & T & T_\circ\ar[l]\ar[d]\\
X'_{2}\ar@<.5ex>[r]\ar@<-.5ex>[r] & X'_{1} && U}$$
such that all natural squares are cartesian.

By \etale{} descent of algebraic spaces, the category $\catC$ is a stack.

Let $\map{j}{U}{\catC}$ be the morphism taking a morphism $\map{h}{T}{U}$ to
the object of $\catC(T)$ given by
\begin{itemize}
\item $T_\circ=T$ and $\map{g_\circ=h}{T_\circ=T}{U}$,
\item $T'_i=U'_i\times_U T$ and $g'_i=j'_i\circ \pi_1\colon T'_i\to
U'_i\to X'_i$ for $i=1,2$.
\end{itemize}
The pull-back of an object $(T_\circ\subseteq T,T'_{\bullet}\to
T,g_\circ,g'_\bullet)$ along $j$ is $(T_\circ\subseteq
T_\circ,T'_{\circ\bullet}\to T_\circ,g_\circ,g'_{\circ\bullet})$ and hence $j$
is an open immersion.

Let $\map{f}{X'}{\catC}$ be the morphism taking a morphism $\map{h}{T}{X'}$ to
the object of $\catC(T)$ given by
\begin{itemize}
\item $T_\circ=h^{-1}(U')$ and $\map{g_\circ=f_U\circ h|_{U'}}{T_\circ}{U}$,
\item $T'_i=T\times_{X',q_1} R'\times_{q_2,X'} X'_i$ for $i=1,2$, and the
morphisms $\map{g'_i=\pi_3}{T'_i}{X'_i}$ for $i=1,2$.
\end{itemize}
In particular, we have a groupoid $\equalizer{T'_1}{T'_2}$ with quotient
$T'=T\times_{X',q_1} R'$ and the induced map $\map{g'}{T'}{X'}$ is
$q_2\circ\pi_2$. We note that there is a section $\map{s=(\id{T},\Delta\circ
h)}{T}{T'}$ and a $2$-morphism $g'\circ s\Rightarrow h\circ\pi_1\circ s=h$.

Let $(S_\circ\subseteq S,S'_{\bullet}\to S,g_\circ,g'_{\bullet})$ be an object
of $\catC(S)$. We will now show that the square
\begin{equation}\label{E:fiber-product}
\vcenter{\xymatrix{S'\ar[r]^{g'}\ar[d] & X'\ar[d]^f\\
S\ar[r] & \catC}}
\end{equation}
is $2$-cartesian.

The first step is to show that it is $2$-commutative. Let $T$ be a scheme and
let $T\to S'$ be a morphism. Let $T_\circ=T\times_{S'} S'_{\circ}$ so that
$T\times_{X'} U'=T_\circ=T\times_{S} S_\circ$. The composition $T\to S'\to S\to
\catC$ corresponds to the object
\begin{equation}\label{E:TS-object}
\vcenter{\xymatrix{T'_{\bullet}\ar[r]\ar[d] & T & T_\circ\ar[l]\ar[d]\\
X'_{\bullet} & & U}}
\end{equation}
where $T'_i=S'_i\times_S T$. The second composition $T\to S'\to X'\to \catC$
corresponds to the object
\begin{equation}\label{E:TX'-object}
\vcenter{\xymatrix{T\times_{X',q_1} R'\times_{q_2,X'} X'_{\bullet}\ar[r]\ar[d]
 & T & T_\circ\ar[l]\ar[d]\\
X'_{\bullet} & & U.}}
\end{equation}
We have the following bi-cartesian squares
$$\vcenter{\xymatrix{U'\ar[r]\ar[d]_{\Delta_{U'/U}} & X'\ar[d]_{\Delta}\\
U'\times_U U'\ar[r] & R'}}\quad\text{and}\quad
\vcenter{\xymatrix{S'_\circ\ar[r]\ar[d]^{\Delta_{S'_\circ/S_\circ}}
  & S'\ar[d]^{\Delta_{S'/S}}\\
S'_\circ\times_{S_\circ} S'_\circ\ar[r] & S'\times_S S'.}}$$
Taking pull-backs of these squares along $\map{\pi_2}{T\times_{X',q_1} R'}{R'}$
and
$T'=T\times_S S'\to S'\times_S S'$ respectively, gives the bi-cartesian squares
$$\vcenter{\xymatrix{T_\circ\ar[r]\ar[d] & T\ar[d]\\
T_\circ\times_U U'\ar[r] & T\times_{X',q_1} R'}}\quad\text{and}\quad
\vcenter{\xymatrix{T_\circ\ar[r]\ar[d] & T\ar[d]\\
T'_\circ\ar[r] & T'.}}$$
Note that $T'_\circ=T_\circ\times_U U'$ so that by the universal property of
pushouts there is an isomorphism of stacks $T\times_{X',q_1} R'\iso T'$.
It follows that there is a $2$-morphism
$$(T\to S'\to X'\to \catC)\Rightarrow (T\to S'\to S\to \catC)$$
and hence the diagram \eqref{E:fiber-product} is $2$-commutative.

To show that the diagram is $2$-cartesian, let $\map{k}{T}{S}$
and $\map{h}{T}{X'}$ be morphisms with a given $2$-morphism $(T\to X'\to
\catC)\Rightarrow (T\to S\to
\catC)$. The morphisms $T\to S\to \catC$ and $T\to X'\to\catC$ correspond to
objects as described in~\eqref{E:TS-object} and~\eqref{E:TX'-object} and the
$2$-morphism gives a cartesian diagram
\begin{equation}
\vcenter{\xymatrix{T\times_{X',q_1} R'\times_{q_2,X'} X'_{\bullet}\ar[r]\ar[d]^{\iso}
 & T\ar@{=}[d] & T_\circ\ar[l]\ar@{=}[d]\\
T'_{\bullet}\ar[r]\ar[d]_{k'_{\bullet}}
 & T\ar[d]_k & T_\circ\ar[l]\ar[d]_{k_\circ}\\
S'_{\bullet}\ar[r]\ar[d]_{g'_{\bullet}} & S & S_\circ\ar[l]\ar[d]_{g_\circ}\\
X'_{\bullet} & & U.}}
\end{equation}
The section $\map{s=(\id{T},\Delta\circ h)}{T}{T\times_{X',q_1} R'}$ induces a
canonical section $\map{s}{T}{T'}$ and a $2$-morphism $g'\circ k'\circ
s\Rightarrow h$ that fits into the diagram
$$\xymatrix{T\ar[dr]_{k'\circ s}\drruppertwocell^{h}{^<-1>}\ddrlowertwocell_{k}{\omit\circ}\\
& S'\ar[r]^{g'}\ar[d]\drtwocell<\omit>{} & X'\ar[d]^f\\
& S\ar[r] & \catC.}$$
This shows that the square is $2$-cartesian. Note that $X'\times_\catC X'\iso
R'$ as asserted in the beginning of the proof.

It follows that the morphism $\map{f\amalg j}{X'\amalg U}{\catC}$ is \etale{}
and surjective and that $f$ is an \etale{} neighborhood of $\catC\setminus
U$. Indeed, the pull-back of $f\amalg j$ along $S\to\catC$ is $S'\amalg
S_\circ\to S$ and this morphism is \etale{} and surjective and $S'\to S$ is an
\etale{} neighborhood. In particular, $\catC$ admits a smooth presentation
and is hence algebraic.

Finally we deduce from Theorem~\tref{T:ETNBHD-IS-PUSHOUT} that $\catC$ is the
pushout of $f$ and $j'$. As the pull-back of an \etale{} neighborhood is
an \etale{} neighborhood, the pushout commutes with arbitrary base change. The
remaining properties listed in Theorem~\tref{T:EXIST-PUSHOUTS} is part of
Proposition~\pref{P:etnbhd-properties}.
\end{proof}

\begin{remark}
If $j'$ is quasi-compact and $U$ and $X'$ are quasi-separated, then
$X=\catC$ is \emph{a posteriori} a quasi-separated stack. The reader who does
not want to introduce a priori non quasi-separated stacks in the proof can
verify directly that when $j'$ is quasi-compact, then
$$(U\amalg X')\times_{\catC} (U\amalg X')\to (U\amalg X')\times (U\amalg X')$$
is indeed representable, quasi-compact and quasi-separated so that
$\catC$ is a quasi-separated stack.
%
\end{remark}

\begin{proposition}\label{P:cartesian-cube}
Given a cartesian diagram of algebraic stacks
$$\xymatrix{U\ar[d] & U'\ar[l]_{f_U}\ar[r]^{j'}\ar[d] & X'\ar[d]\\
V & V'\ar[l]_{g_V}\ar[r]^{k'} & Y'}$$
such that $j'$ and $k'$ are open immersions and $f_U$ and
$g_V$ are \etale{}, let $X=X'\amalg_{U'} U$ and $Y=Y'\amalg_{V'} V$ denote the
pushouts. Then every face of the induced cube
$$\xymatrix@!C=1pc@!R=0.7pc{U'\ar[rr]^(.3){j'}\ar[rd]\ar[dd]_(.3){f_U}
                && X'\ar[rd]\ar[dd]_(.3){f}|\hole \\
& V'\ar[rr]_(.7){k'}\ar[dd]^(.7){g_V} && Y'\ar[dd]^(.7){g}\\
U\ar[rr]^(.3){j}|\hole\ar[rd] && X\ar[rd]\\
& V\ar[rr]_(.7){k} && Y}$$
is cartesian.
\end{proposition}
\begin{proof}
We may replace $Y$, $Y'$, $V$ and $V'$ by their pull-backs along $X\to Y$ and
assume that $X=Y$. Since $f$, $g$, $j$ and $k$ are \etale{}, it is enough to
verify that the cube is cartesian over points of $Y$. We can thus assume that
$X=Y=\Spec k$. But then either $V=Y$ which implies that the top and the bottom
square are trivial and $f=g=f_U=g_V$, or $V=\emptyset$ which implies that
$U=U'=V=V'=\emptyset$ and $X=X'=Y=Y'$. Thus in either case, we have that $U=V$
and $X'=Y'$.
\end{proof}

\begin{corollary}\label{C:effective-etnbhd-descent-of-stacks}
Let $\map{j}{U}{X}$ be an open immersion of algebraic stacks and let
$\map{f}{X'}{X}$ be an \etale{} neighborhood of $X\setminus U$. Let
$U'=f^{-1}(U)$. Let $W_U\to U$ and $W'\to X'$ be morphisms of stacks and
let $W'|_{U'}\iso W_U\times_U U'$ be an isomorphism. Then there is a stack
$W\to X$, unique up to unique $2$-morphism, and morphisms $W'\to W$
and $W_U\to W$ such that every face of the cube
$$\xymatrix@!C=1pc@!R=0.7pc{W'|_{U'}\ar[rr]^(.3){}\ar[rd]\ar[dd]_(.3){}
                && W'\ar[rd]\ar[dd]_(.3){}|\hole \\
& U'\ar[rr]_(.7){j'}\ar[dd]^(.7){f|_U} && X'\ar[dd]^(.7){f}\\
W_U\ar[rr]^(.3){}|\hole\ar[rd] && W\ar[rd]\\
& U\ar[rr]_(.7){j} && X}$$
is cartesian.
\end{corollary}
\begin{proof}
Any stack $W\to X$ satisfying the condition of the Corollary is a pushout of
$W'|_{U'}\to W'$ and $W'|_{U'}\to W_U$ by Theorem~\tref{T:ETNBHD-IS-PUSHOUT}.
By Theorem~\tref{T:EXIST-PUSHOUTS}, the pushout $W$ exists and the cube
is cartesian by Proposition~\pref{P:cartesian-cube}.
\end{proof}

\end{section}


\begin{section}{Constructible sheaves}\label{S:constructible}
In this section, we show that given a constructible sheaf $\sF$ on a
quasi-compact and quasi-separated stack $X$, there is a finite filtration of
$X$ in open quasi-compact substacks such that $\sF$ is locally constant on the
induced stratification of $X$. We begin with a short review of constructible
sheaves on stacks, cf.\ \cite[Exp.~IX, \S2]{sga4} and~\cite[Ch.~18]{laumon}.

\begin{xpar}
Let $X$ be a quasi-compact and quasi-separated stack. Recall that a subset
$W\subseteq |X|$ is locally closed if $W$ is the intersection of a closed and an
open subset. A locally closed subset $W$ is constructible if and only if
$W$ and its complement are quasi-compact, or equivalently, if and only if
$W=U\setminus V$ where $V\subseteq U\subseteq X$ are open and quasi-compact.
\end{xpar}

\begin{xpar}
Let $X$ be a quasi-compact and quasi-separated \emph{scheme}. Let $\sF$ be a
sheaf of sets on the small \etale{} site of $X$. Recall that $\sF$ is
\emph{locally constant} if there exists a covering $\{U_i\to X\}$ such that
$\sF|_{U_i}$ is a constant sheaf for every $U_i$. The sheaf $\sF$ is
\emph{constructible} if there exists a finite stratification $|X|=\bigcup W_i$
into locally closed constructible subsets $W_i\subseteq |W|$ such that
$\sF|_{W_i}$ is
locally constant and finite~\cite[Exp.~IX, Prop.~2.4]{sga4}. Note that the
choice of scheme structure on the $W_i$'s is irrelevant since the \etale{}
sites of $W_i$ and $(W_i)_\red$ are equivalent.

Every constructible sheaf is represented by an algebraic space, \etale{} and
finitely presented over $X$~\cite[Exp.~IX, Prop.~2.7]{sga4}. In other words,
there is a one-to-one correspondence between constructible sheaves on $X$ and
finitely presented \etale{} morphisms $X'\to X$ of algebraic spaces given by
taking $X'$ to the corresponding sheaf of sections. Note that $X'$ is a scheme
if $X'\to X$ is separated~\cite[Cor.~6.17]{knutson_alg_spaces}. A constructible
sheaf is locally constant if and only if it is represented by a finite \etale{}
morphism.
\end{xpar}

\begin{xpar}
Let $X$ be a quasi-compact and quasi-separated \emph{stack} and let
$\map{\pi}{V}{X}$ be a presentation such that $V$ is a quasi-compact and
quasi-separated scheme (e.g., an affine scheme). Since $\pi$ is open,
surjective and quasi-compact, it follows that a subset $W\subseteq |X|$ is
locally closed (resp.\ locally closed and constructible) if and only if
$\pi^{-1}(W)$ is so. By definition, a sheaf of sets $\sF$ on the lisse-\etale{}
site of $X$, is locally constant (resp.\ constructible) if it is cartesian and
$\pi^{*}\sF$ is locally constant (resp.\ constructible)~\cite[D\'ef.\
18.1.4]{laumon}. This definition does not depend on the choice of
presentation. It follows, e.g., using local constructions as in~\cite[Ch.\
14]{laumon}, that the category of constructible sheaves on $X$ is equivalent to
the category $\Stack_{\repr,\fp,\metale/X}$ of representable finitely presented
and \etale{} morphisms $X'\to X$.
\end{xpar}

Surprisingly, the following result (closely related
to~\cite[Prop.~18.1.7]{laumon} and~\cite[Exp.~IX, Prop.~2.5]{sga4}) seems to be
missing in the literature.

\begin{proposition}\label{P:filtration-of-constructible}
Let $X$ be a quasi-compact and quasi-separated \emph{stack}.
\begin{enumerate}
\item Let $\sF$ be a lisse-\etale{} sheaf of sets on $X$. Then $\sF$ is
constructible if and only if there exists a finite filtration
$\emptyset=X_0\subset X_1\subset X_2\subset \dots \subset X_n=X$ of open
quasi-compact subsets such that $\sF|_{X_{i}\setminus X_{i-1}}$ is locally
constant of constant finite rank for every $i=1,2,\dots,n$.\label{PI:filtr:cons}
\item Let $\map{f}{X'}{X}$ be a representable \etale{} morphism.
Then $f$ is of finite presentation if and only if
there exists a filtration of $X$ as in \ref{PI:filtr:cons} such that
$f|_{X_{i}\setminus X_{i-1}}$ is finite and \etale{} of constant rank for every
$i=1,2,\dots,n$.
\label{PI:filtr:repr-fp-et}
\end{enumerate}
\end{proposition}
\begin{proof}
The two statements are equivalent. We will show the Proposition in the
form~\ref{PI:filtr:cons}. The condition is clearly sufficient. To prove
necessity, let $\map{\pi}{V}{X}$ be a presentation with $V$ a quasi-compact
and quasi-separated scheme. If there exists a filtration
$\emptyset=V_0\subseteq V_1\subseteq \dots \subseteq V_n=V$ of open
quasi-compact subsets such that $\sF|_{V_{i}\setminus V_{i-1}}$ is locally
constant of constant finite rank, then the filtration of $X$ given by
$X_i=\pi(V_i)$ suffices. Replacing $X$ with $V$, we can thus assume that
$X$ is a scheme.

By definition, there is then a stratification $|X|=\bigcup W_i$ into locally
closed constructible subsets such that $\sF|_{W_i}$ is locally constant and by
refining the $W_i$'s, we can assume that the rank of $\sF|_{W_i}$ is constant.
Write $W_i=U_i\setminus V_i$ where $V_i\subseteq U_i\subseteq X$ are
quasi-compact open subsets. Let $T$ be the topology on $X$ generated by all
$U_i$'s and $V_i$'s. Then every element of $T$ is a quasi-compact open subset
and $T$ is finite. Let $N=|T|$ be the number of open subsets. We will
construct a filtration with elements of $T$ by induction on $N$.

Let $X_1\in T$ be a non-empty minimal open
subset. Then $\sF|_{X_1}$ is locally constant of constant rank. By induction,
we have a filtration $\emptyset=Z_1\subset Z_2\subset\dots\subset Z_n=Z=
X\setminus X_1$ with $Z_i\in T|_Z$ such that $\sF|_{Z_i\setminus Z_{i-1}}$
is locally constant of constant finite rank. As $X_1\cup Z_i\in T$,
we obtain a filtration as in the proposition by taking $X_i=X_1\cup Z_i$ for
$i=1,\dots,n$.
\end{proof}

\begin{remark}
In general, the stratification in
Proposition~\pref{P:filtration-of-constructible} is not canonical. There are
two important special cases though:
\begin{enumerate}
\item Let $X$ be a quasi-compact and quasi-separated stack and let
$\map{f}{X'}{X}$ be a \emph{separated} and quasi-compact \etale{}
morphism. Then the fiber rank of $f$ is a constructible and lower
semi-continuous function. Thus, there is a
\emph{canonical} finite filtration
$X=X_0\supseteq X_1\supseteq X_2\supseteq\dots\supseteq X_{n+1}=\emptyset$ of
$X$
into open quasi-compact subsets, such that $f$ is finite and \etale{} of
constant rank $i$ over the locally closed constructible subset $X_i\setminus
X_{i+1}$.

Similarly, if $f$ is \emph{universally closed} and finitely presented but not
necessarily separated, the fiber rank of $f$ is constructible and upper
semi-continuous and we obtain a canonical filtration.
\item Let $X$ be a \emph{noetherian} stack and let $\sF$ be a constructible
sheaf. Let $U$ be the maximal open subset such that $\sF|_{U}$ is locally
constant. We can then take $X_1$ as the open and closed subset of $U$ with
minimal fiber rank. Proceeding with $X\setminus X_1$ we obtain a canonical
filtration by noetherian induction. If $X$ is not noetherian, then this
procedure would also give a finite filtration but the $X_i$'s would not
necessarily be quasi-compact.
\end{enumerate}
\end{remark}

\end{section}


\begin{section}{\Etale{} \devissage{}}\label{S:etale-devissage}
In this section, we prove the \devissage{} theorem for representable finitely
presented \etale{}
morphisms. In the separated case, this \devissage{} was used by Raynaud and
Gruson
to pass from algebraic spaces to schemes~\cite[\S5.7]{raynaud-gruson}. We have
taken some care to also include the non-separated case. This is motivated by
Examples~\pref{E:etale-envelope} and~\pref{E:connected-fibration}. The starting
point is the existence of stratifications of finitely presented \etale{}
morphisms as in Proposition~\pref{P:filtration-of-constructible}.
The idea is to then use symmetric products to pass to \etale{} neighborhoods.

\begin{definition}
Let $\map{f}{X'}{X}$ be a representable morphism of stacks. We let $(X'/X)^d$
be the $d$\textsuperscript{th} fiber product of $X'$ over $X$ and we let the
symmetric group $\SG{d}$ act on $(X'/X)^d$ by permuting the factors. Let
$Z\subset |X|$ be a closed subset of $X$ such that $f|_{Z_\red}$ is
\emph{separated} and let
$Z'=f^{-1}(Z)$. Further let $\Delta_{Z'/Z}$ be the diagonal of $Z'\times_Z Z'$
as a closed subset of $|X'\times_X X'|$ and let $\Delta(Z)$ be the big
diagonal of $Z'$ in $(X'/X)^d$, i.e., the $\SG{d}$-orbit of
$\Delta_{Z'/Z}\times_X (X'/X)^{d-2}$. Then $\Delta(Z)$ is closed and we let
$\SEC^d_Z(X'/X)\subseteq (X'/X)^d$ be its complement. We let
$\ET^d_Z(X'/X)=[\SEC^d_Z(X'/X)/\SG{d}]$ be the stack quotient.
\end{definition}

\begin{remark}
The stack $\SEC^d_Z(X'/X)$ parameterizes $d$ sections of $\map{f}{X'}{X}$ 
such that these sections are disjoint over $Z$. The stack $\ET^d_Z(X'/X)$
parameterizes finite \etale{} morphisms $W\to X$ of rank $d$ together with
an $X$-morphism $W\to X'$ that is a closed immersion over $Z\inj X$.
If $f$ is separated, then we can form $\ET^d(X'/X):=\ET^d_X(X'/X)$ which is the
stack considered in~\cite[6.6]{laumon}. If $X'/X$ is \etale{} of constant rank
$d$, then $\ET^d(X'/X)\to X$ is an isomorphism.
\end{remark}

\begin{lemma}\label{L:ET-neighborhoods}
Let $\map{f}{X'}{X}$ be a representable \etale{} surjective morphism of
algebraic stacks, and let $Z\subset |X|$ be a closed subset such that
$f|_{Z_\red}$ is
finite of constant rank $d$.
\begin{enumerate}
\item The projections $\map{\pi_1,\pi_2,\dots,\pi_d}{\SEC^d_Z(X'/X)}{X'}$ are
\etale{} and surjective.
\item $\ET^d_Z(X'/X)\to X$ is a surjective \etale{} neighborhood of $Z$.
\item If $f$ is separated then
$\ET^d(X'/X)\to X$ is a representable and separated \etale{} neighborhood of
$Z$.
\end{enumerate}

\end{lemma}
\begin{proof}
That $\pi_i$ and $\ET^d_Z(X'/X)\to X$ are \etale{} and surjective follows from
the construction. We also have
$\ET^d_Z(X'/X)|_Z=\ET^d(f^{-1}(Z)/Z)\iso Z$
so that the second assertion holds. If $f$ is
separated then $\SG{d}$ acts freely on $\SEC^d_X(X'/X)$, relative to $X$, so
that $\ET^d(X'/X)\to X$ is representable.
\end{proof}

\begin{proof}[Proof of Theorem~\tref{T:DEVISSAGE}]
Let $\map{f}{X'}{X}$ be representable, \etale{} and surjective of finite
presentation such that $X'\in \catD$. Pick a filtration $\emptyset=X_0\subset
X_1\subset X_2\subset \dots \subset X_n=X$ of open quasi-compact subsets such
that $f|_{X_{i}\setminus X_{i-1}}$ is finite and \etale{} of constant rank
$d_i$ as in Proposition~\pref{P:filtration-of-constructible}. Let
$X'_i=f^{-1}(X_i)$. We will show that $X\in \catD$ by induction. Thus, let
$1\leq i\leq n$ and assume that $X_{i-1}\in \catD$.

Let $Z_i=X_i\setminus X_{i-1}$. If $f$ is separated, let
$\SEC_i=\SEC^{d_i}(X'_i/X_i)$ and $\ET_i=\ET^{d_i}(X'_i/X_i)$. If $f$ is not
separated, let $\SEC_i=\SEC^{d_i}_{Z_i}(X'_i/X_i)$ and
$\ET_i=\ET^{d_i}_{Z_i}(X'_i/X_i)$.
By Lemma~\pref{L:ET-neighborhoods} we have that $\ET_i\to X_{i}$ is an \etale{}
neighborhood of $Z_i$, that $\SEC_i\to \ET_i$ is a finite
\etale{} surjective morphism of rank $d_i!$ and that there is a finitely
presented \etale{} morphism $\SEC_i\to X'_i$. Moreover, if $f$ is separated,
then $\ET_i\to X_{i}$ is representable and separated.

By (D1) we have that $\SEC_i\in \catD$, by (D2) it follows that
$\ET_i\in \catD$ and by (D3) we have that $X_i\in \catD$. It follows that
$X\in \catD$ by induction.
\end{proof}

\begin{remark}[Sheaf interpretation]
Condition (D1) of Theorem~\tref{T:DEVISSAGE} states that $\catD$ is a sieve on
$\catE$, i.e., a presheaf $\catE^{\op}\to \{\emptyset,\{\ast\}\}$. Conditions
(D2) and (D3) signify that this presheaf satisfies the sheaf condition with
respect to finite and surjective \etale{} morphisms and with respect to
coverings of the form $(U\to X,\;X'\to X)$ where $U\to X$ is an open immersion
and $X'\to X$ is an \etale{} neighborhood of $X\setminus U$. The conclusion of
Theorem~\tref{T:DEVISSAGE} is that the presheaf satisfies the sheaf condition
with respect to representable \etale{} coverings. Theorem~\tref{T:DEVISSAGE}
can be generalized to arbitrary presheaves.
\end{remark}

\begin{remark}[Nisnevich topology]\label{R:Nisnevich}
A Nisnevich covering of a noetherian scheme $X$ is a surjective family of
\etale{} morphisms $\map{p_i}{X_i}{X}$ such that every point
$\map{x}{\Spec \kappa(x)}{X}$ admits a lifting to some $X_i$. The Nisnevich
topology is the topology associated to the pretopology of Nisnevich coverings.
A well-known and easy result is that the Nisnevich topology is generated by
coverings of the form $(U\to X,\;X'\to X)$ where $U\to X$ is an open immersion
and $X'\to X$ is an \etale{} neighborhood of $X\setminus U$.

It immediately follows from Theorem~\tref{T:DEVISSAGE} that the \etale{}
topology on a scheme is the topology generated by the finite \etale{} topology
and the Nisnevich topology. Indeed, if $\catD\subseteq \catE$ is a covering
sieve in the \etale{} topology, then the sheafification of $\catD$ in the
topology generated by the finite \etale{} topology and the Nisnevich topology
is $\catE$ by Theorem~\tref{T:DEVISSAGE} so that $\catD$ is covering in this
topology.
\end{remark}

We end this section with some examples:

\begin{example}[{cf.\ \cite[5.7.6]{raynaud-gruson}}]
Let $X$ be a quasi-compact and quasi-separated algebraic space. Then there
exists an affine scheme $X'$ and an \etale{} presentation
$\map{f}{X'}{X}$. Since $f$ is separated, the fiber rank of $f$ is a lower
semi-continuous constructible function. Thus, there is a \emph{canonical}
filtration $\emptyset=X_{n+1}\subseteq X_n\subseteq X_{n-1}\subseteq\dots\subseteq
X_1=X$ of quasi-compact open substacks $X_i$ such that $f|_{X_d\setminus
X_{d+1}}$ is finite of constant rank $d$. Let $\ET_d=\ET^d(X'_d/X_d)$ so that
$\ET_d\to X_d$ is a representable, separated and surjective \etale{}
neighborhood of $X_d\setminus X_{d+1}\inj X_d$. As $\ET_d$ is the quotient of
the quasi-affine scheme $\SEC^d(X'_d/X_d)$ by a \emph{free} group action, it is
a
quasi-affine scheme by Lemma~\pref{L:finite-flat-pres-of-alg-space:aff/qaff}.
In particular, we have that $X_d\setminus X_{d+1}$ is a quasi-affine scheme.
\end{example}


\begin{example}[{\cite{rydh_embeddings-of-unramified}}]\label{E:etale-envelope}
Let $\map{f}{X}{Y}$ be an unramified morphism of algebraic stacks. Then there
is a canonical factorization $X\inj E_{X/Y}\to Y$ of $f$ where $X\inj E_{X/Y}$
is a closed immersion and $\map{e}{E_{X/Y}}{Y}$ is \etale{}. The \etale{}
morphism $e$ is almost never separated but $e$ is at least universally
closed if $f$ is finite. If $f$ is representable and of finite presentation,
then so is $e$. The \devissage{} method can thus be extended to treat
representable
finitely presented unramified morphisms.
\end{example}

\begin{example}[{\cite[6.8]{laumon} and \cite[Thm.~2.5.2]{romagny_components-in-families}}]\label{E:connected-fibration}
Let $\map{f}{X}{Y}$ be a flat and finitely presented morphism with geometrically
reduced fibers (e.g., $f$ smooth). Then $f$ has a
canonical factorization $X\to \pi_0(X/Y)\to Y$ where $X\to \pi_0(X/Y)$
is surjective with geometrically connected fibers and
$\map{\pi}{\pi_0(X/Y)}{Y}$ is \etale{} and representable. If $f$ is
\etale{} and representable then $f=\pi$ so that $\pi$ is separated if and only
if $f$ is
separated. On the other hand, there are examples where $f$ is smooth and
separated but $\pi$ is not separated.
\end{example}

\end{section}


\begin{section}{Quasi-finite flat \devissage{}}\label{S:qff-devissage}
Let $S$ be an algebraic stack. We let $\Stack_{\fp,\qff/S}$ denote the category
of quasi-finite and flat morphisms $X\to S$ of finite presentation and let
$\Stack_{\repr,\sep,\fp,\qff/S}$ denote the subcategory of representable and
separated morphisms. The main result of this section is the following theorem:

\begin{theorem}\label{T:qff-devissage}
Let $S$ be a quasi-compact and quasi-separated algebraic stack and let $\catF$
be either $\Stack_{\fp,\qff/S}$ or $\Stack_{\repr,\sep,\fp,\qff/S}$. Let
$\catD\subseteq \catF$ be a full subcategory such that
\renewcommand{\theenumi}{{\upshape{(D\arabic{enumi})}}}
\begin{enumerate}
\item If $X\in \catD$ and $(X'\to X)\in \catF$ is \etale{} then $X'\in \catD$.
\item If $X'\in \catD$ and $(X'\to X)\in \catF$ is finite and surjective,
  then $X\in \catD$.
\item If $\map{j}{U}{X}$ and $\map{f}{X'}{X}$ are morphisms in $\catF$ such that
  $j$ is an open immersion and $f$ is an \etale{} neighborhood of $X\setminus
  U$, then $X\in \catD$ if $U,X'\in \catD$.
\end{enumerate}
Then if $(X'\to X)\in \catF$ is representable, locally separated and
surjective and $X'\in \catD$, we have that $X\in \catD$.
%
\end{theorem}

Note that the only difference between the conditions of
Theorem~\tref{T:DEVISSAGE} and Theorem~\pref{T:qff-devissage} is that in
the second condition $X'\to X$ is only required to be flat, not merely \etale{}.

\begin{remark}\label{R:qff-devissage-of-qf-stack}
Let $S$ be a quasi-compact and quasi-separated algebraic stack with
quasi-finite diagonal. Then $S$ admits a quasi-finite flat presentation
$S'\to S$ by Theorem~\pref{T:qf-cover-of-Artin}. If in addition $S$ has
\emph{locally separated}
diagonal, then we can arrange so that $S'\to S$ is locally separated. If
$S'\in \catD$
we can then apply Theorem~\pref{T:qff-devissage} to deduce that $S\in \catD$.
\end{remark}

Theorem~\pref{T:qff-devissage} is an immediate corollary of
Theorem~\tref{T:DEVISSAGE} and the following result about the \etale{}-local
structure of quasi-finite morphisms.

\begin{theorem}\label{T:qff:finite}
Let $\map{f}{X}{Y}$ be a quasi-finite flat morphism of finite
presentation between algebraic stacks such that $\Delta_f$ is unramified (e.g.,
$f$ representable). Then there exists a commutative diagram
$$\xymatrix{X'\ar[d]_{f'}\ar[r] & X\ar[d]_{f}\\
Y'\ar[r] & Y\ar@{}[ul]|\circ}$$
where the horizontal morphisms are \etale{} and quasi-separated,
where $X'\to X$ is surjective and where $\map{f'}{X'}{Y'}$ is finite, flat and
of finite presentation. Moreover,
\begin{enumerate}
\item If $f$ is representable and separated, we can arrange so that
the horizontal morphisms are representable, separated and of finite
presentation.
\item If $f$ is representable and locally separated, we can arrange so that
the horizontal morphisms are representable.
\item\label{TI:qff-f:fp}
If $Y$ is quasi-compact and quasi-separated, we can arrange so that the
horizontal morphisms are of finite presentation.
\end{enumerate}
\end{theorem}

Note that if $Y$ is a quasi-compact Deligne--Mumford stack with quasi-compact
and separated diagonal, then the result follows from the well known case where
$X$ and $Y$ are schemes. The proof in the general case is more subtle
and is inspired by Keel and Mori's usage of Hilbert
schemes~\cite[\S4]{keel_mori_quotients}.

\begin{proof}
Let us first assume that $f$ is representable and separated. The \etale{} sheaf
$f_!\left(\underline{\Z/2\Z}_X\right)$ is \emph{constructible} and hence
represented by a finitely presented \etale{} morphism $Y'\to Y$. We let
$X'\subset X\times_Y Y'$ be the support of the universal section. This is an
open and closed subset so that $X'\to X$ is finitely presented and \etale{}.
By the definition of $f_!$ we have that $\map{f'}{X'}{Y'}$ is proper and hence
finite. That
$X'\to X$ is surjective can be checked after passing to fibers of $f$ since
$f_!$ commutes with arbitrary base change.

To see that $Y'\to Y$ is separated, we describe $Y'$ as a Hilbert scheme. Let
$\Hilb^{\mopen}_k(X/Y)$ be the open subscheme of the relative Hilbert scheme of
$k$ points on $X/Y$ parameterizing open and closed subschemes. That is, for any
scheme $T$ and morphism $T\to Y$, the $T$-points of $\Hilb^{\mopen}_k(X/Y)$ are
in bijection with open and closed subschemes $Z\inj X\times_Y T$ such that
$Z\to T$ is flat and finite of constant rank $k$. Then $Y'=\coprod_{k\geq 0}
\Hilb^{\mopen}_k(X/Y)$ so that $Y'\to Y$ is separated.

We now drop the assumption that $f$ is representable and separated. It is still
possible to
define $f_!\left(\underline{\Z/2\Z}_X\right)$ but it does not carry a universal
section. Instead, consider the \emph{Hilbert stack} $\HilbSt^{\metale}_k(X/Y)$
parameterizing flat families $Z\to T$ of constant rank $k$ with an \etale{}
morphism $Z\to X\times_Y T$. The Hilbert stack $\HilbSt^{\metale}_k(X/Y)$ is an
open substack of the full Hilbert stack $\HilbSt_k(X/Y)$ which is known to be
algebraic and of finite presentation~\cite[Thm.~4.4]{rydh_hilbert}.
It is readily verified that $\HilbSt^{\metale}_k(X/Y)\to Y$ is \etale{}.
We let $\HilbSt=\coprod_{k\geq 0} \HilbSt^{\metale}_k(X/Y)$. In the general
case we can then let $Y'=\HilbSt$ and let $X'\to Y'$ be the universal family.
To verify that $X'\to X$ is surjective, we can assume that $Y$ is the spectrum
of a field so that $X$ is a Deligne--Mumford stack. It is then obvious that
$X'\to X$ is surjective since $X$ locally admits an \etale{} presentation
$Z\to X$ with $Z\to X\to Y$ finite and flat.

If $f$ is representable and \emph{locally separated}, then we let $Y'$ be the
largest
open substack of $\HilbSt$ such that $Y'\to Y$ is representable and let $X'$ be
the restriction of the universal family to $Y'$. It remains to verify that
$X'\to X$ is surjective.

The diagonal of $\HilbSt$, relative to $Y$, is of finite presentation
(and \etale{} and
separated) so the locus $R\subseteq |\HilbSt|$ where the inertia stack
$I_{\HilbSt/Y}\to \HilbSt$ is an
isomorphism is a constructible subset (and a closed subset but we do not use
this). Since $Y'$ is the interior of $R$, it follows that the construction of
$Y'$ commutes with flat base change on $Y$ and that a point
$\map{h}{\Spec k}{\HilbSt}$ is in $Y'$ if and only if the fibers of
$I_{\HilbSt/Y}\to
\HilbSt$ over $h$ and \emph{its generizations} have rank $1$~\cite[Thm.~7.3.1
and Prop.~7.3.3]{egaI_NE}.

To show that a point $\map{x}{\Spec k}{X}$ is in the image of $X'\to X$, we
can thus assume that $Y$ is the spectrum of a strictly henselian local
ring and that $x$ lies in the fiber of the closed point of $Y$. Then by
Lemma~\pref{L:qf+loc-sep-over-strictly-local} we have that $x$ lies in an
open subscheme $Z\subseteq X$ that is finite over $Y$. The family $Z\to Y$
induces a morphism $Y\to Y'\subset \HilbSt$ so that $x$ is in the image
of $X'\to X$.

Finally to show~\ref{TI:qff-f:fp} it is enough to replace $Y'$ with a
quasi-compact open
substack and $X'$ with its inverse image.
\end{proof}


%

\end{section}


\begin{section}{Stacks with quasi-finite diagonals}\label{S:qf-stacks}
In this section we show that every stack with quasi-finite diagonal has a
locally quasi-finite flat presentation. The main purpose of this result is to
show that the quasi-finite flat \devissage{}, Theorem~\pref{T:qff-devissage},
can
indeed be applied to a presentation of a stack with quasi-finite and locally
separated diagonal as mentioned in Remark~\pref{R:qff-devissage-of-qf-stack}.
We also combine this result with Theorem~\pref{T:qff:finite}
and deduce that stacks with quasi-finite diagonals
admit finite flat presentations \etale{}-locally.

\begin{theorem}[Quasi-finite presentations]\label{T:qf-cover-of-Artin}
Let $X$ be an algebraic stack with quasi-finite diagonal. Then there is
a locally quasi-finite flat presentation $U\to X$ with $U$ a scheme.
\end{theorem}
\begin{proof}
It is enough to construct for every point $\xi\in |X|$ a locally quasi-finite
flat morphism $\map{p}{U}{X}$ locally of finite presentation with $U$ a scheme
such that $\xi\in p(U)$. Choose an immersion $Z\inj X$ as in
Theorem~\pref{T:algebraicity-of-points} so that $\xi\in|Z|$, $I_Z\to Z$ is
flat and locally of finite presentation and $Z$ is an fppf
gerbe over a scheme
$\underline{Z}$. The diagonal of $Z\to \underline{Z}$ is quasi-finite, flat and
locally of finite presentation. This follows from the diagram
$$\xymatrix{%
I_Z\ar[r]\ar[d] & Z\ar[d]^{\Delta_{Z/\underline{Z}}}\\
Z\ar[r]^-{\Delta_{Z/\underline{Z}}} & Z\times_{\underline{Z}} Z}$$
since the diagonal is covering in the fppf topology.

Let $V\to X$ be a flat (or smooth) presentation of $X$ with $V$ a scheme. Then
$V\times_X Z\to Z\to \underline{Z}$ is flat. Let $\underline{\xi}\in
\underline{Z}$ be the image of $\xi$. We will now do a standard slicing
argument,
cf.~\cite[Prop.~17.16.1]{egaIV}. Let $v$ be a closed point in the fiber
$V_{\xi}:=V\times_X \stG_\xi=
V\times_X Z\times_{\underline{Z}} \Spec \kappa(\underline{\xi})$
at which the fiber is Cohen--Macaulay. Let
$f_1,f_2,\dots,f_n$ be a regular sequence in $\sO_{V_{\xi},v}$ such
that the quotient is artinian. Since $\sO_{V,v}\to \sO_{V_{\xi},v}$
is surjective, we can lift this sequence to a sequence $g_1,g_2,\dots,g_n$ of
global
sections of $\sO_V$ after replacing $V$ with an open neighborhood of $v$.

Let $W\inj V$ be the closed subscheme defined by the ideal
$(g_1,g_2,\dots,g_n)$. Since $f_1,f_2,\dots,f_n$ is regular, it follows
by~\cite[Thm.~11.3.8]{egaIV} that $W\times_X Z\to Z\to\underline{Z}$ is flat in a
neighborhood of $v$. After shrinking $V$ we can thus assume that
$W\times_X Z\to Z\to\underline{Z}$ is flat. Since $W\times_X Z\to
\underline{Z}$ is quasi-finite at $v$, we can also assume that $W\times_X Z\to
Z\to\underline{Z}$ is quasi-finite after further shrinking
$V$~\cite[Cor.~13.1.4]{egaIV}.

Now, since the diagonal of $Z\to \underline{Z}$ is flat and quasi-finite, it
follows that $W\times_X Z\to Z$ is flat and quasi-finite.
Finally, we apply~\cite[11.3.8 and 13.1.4]{egaIV} on $W\inj V\to X$ to deduce
that $W\to X$ is flat and quasi-finite in an open neighborhood of $v\in W$.
\end{proof}

\begin{theorem}[Finite flat presentations]
\label{T:qcqf-stacks:finite-presentations}
Let $X$ be an algebraic stack. The following are equivalent
\begin{enumerate}
\item\label{TI:qfin-diag}
$X$ is quasi-compact and quasi-separated with quasi-finite
(resp.\ quasi-finite and locally separated, resp.\ quasi-finite and separated)
diagonal.
\item\label{TI:qfin-pres}
There exists a quasi-finite flat presentation $\map{p}{U}{X}$ with $U$
affine and such that $p$ is finitely presented (resp.\ finitely presented
and locally separated, resp.\ finitely presented and separated).
\item\label{TI:fin+etale-pres}
There exists an \etale{} (resp.\ representable \etale{}, resp.\ representable,
separated and \etale{}) surjective morphism $X'\to X$ of finite presentation
such that $X'$ admits a finite flat presentation $V\to X'$ with $V$ a
quasi-affine scheme.
\end{enumerate}
\end{theorem}
\begin{proof}
Clearly \ref{TI:fin+etale-pres}$\implies$\ref{TI:qfin-pres}.
That \ref{TI:qfin-diag}$\implies$\ref{TI:qfin-pres} follows from
Theorem~\pref{T:qf-cover-of-Artin} and that \ref{TI:qfin-pres}$\implies$\ref{TI:qfin-diag} follows from Lemmas~\pref{L:diag-separated:descent}
and~\pref{L:quasi-separatedness}.
Assume that \ref{TI:qfin-diag} holds and choose a quasi-finite flat presentation
$\map{p}{U}{X}$ with $U$ affine as in \ref{TI:qfin-pres}. By
Theorem~\pref{T:qff:finite},
there is a commutative diagram
$$\xymatrix{U'\ar[d]^{p'}\ar[r] & U\ar[d]^{p}\\
X'\ar[r] & X\ar@{}[ul]|\circ}$$
such that $X'\to X$ and $U'\to U$ are \etale{}
(resp.\ representable and \etale{}, resp.\ representable, separated and
\etale{}) and surjective of finite presentation and
$\map{p'}{U'}{X'}$ is finite and faithfully flat.

If $X$ is arbitrary (resp.\ has locally separated diagonal), then $U'$ is
a Deligne--Mumford stack (resp.\ an algebraic space), so that $X'$ has locally
separated (resp.\ separated) diagonal by Lemma~\pref{L:diag-separated:descent}.
We can thus replace $X$ with $X'$ and assume that $X$ has locally separated
(resp.\ separated) diagonal. Applying this argument twice,
we can assume that $X$ has separated diagonal.

If $X$ has separated diagonal, then $X'\to X$ and $U'\to U$ are separated and
hence quasi-affine by Zariski's Main
Theorem~\cite[Thm.~II.6.15]{knutson_alg_spaces} and the theorem follows.
\end{proof}

\begin{remark}
The proof of~\cite[Exp.~VIII, Cor.~7.6]{sga1} shows that
in~\ref{TI:fin+etale-pres} we can choose $V\to X'\to X$ such that $V$ is
affine. Also see~\cite[Exp.~V, p.~270]{sga3}.
\end{remark}

\end{section}


\appendix

\begin{section}{Separation axioms for algebraic stacks}\label{A:separation}
A sheaf of sets $F$ on the category of schemes $\Sch$ with the \etale{}
topology is an \emph{algebraic space} if there exists a scheme $X$ and a
morphism $X\to F$ that is represented by surjective \etale{} morphisms
of schemes~\cite[D\'ef.~5.7.1]{raynaud-gruson}, i.e., for any scheme $T$
and morphism $T\to F$, the fiber product $X\times_F T$ is a scheme and
$X\times_F T\to T$ is surjective and \etale{}.

A \emph{stack} is a category fibered in groupoids over $\Sch$ with the
\etale{} topology satisfying the usual sheaf condition~\cite{laumon},
or equivalently, a $2$-sheaf $\Sch\to \Grpd$ in the sense of
Appendix~\ref{A:2-sheaves}.
A morphism $\map{f}{X}{Y}$ of stacks is \emph{representable} if for any scheme
$T$ and morphism $T\to Y$, the $2$-fiber product $X\times_Y T$ is an algebraic
space. A stack $X$ is \emph{algebraic} if there exists a smooth presentation,
i.e., a smooth, surjective and representable morphism $U\to X$ where $U$ is a
scheme (or algebraic space).

\begin{definition}
A morphism of algebraic stacks $\map{f}{X}{Y}$ is \emph{quasi-separated} if
$\Delta_f$ and $\Delta_{\Delta_f}$ are quasi-compact. An algebraic stack $X$ is
quasi-separated if $X\to \Spec \Z$ is quasi-separated. A morphism of algebraic
stacks is of \emph{finite presentation} if it is locally of finite presentation,
quasi-compact and quasi-separated.
\end{definition}

Recall that in~\cite{laumon} algebraic stacks are by definition quasi-separated
and have separated diagonals. In the remainder of this appendix we give
criteria for when this is the case.
An important example of a stack with non-separated diagonal is the stack of
log structures~\cite{olsson_log-geom-and-alg-stacks}. On the other hand, this
stack has at least \emph{locally separated diagonal}.

\begin{definition}
Let $\map{f}{X}{Y}$ be a representable morphism. We say that $f$ is
\emph{locally separated} if $\Delta_f$ is an immersion.
\end{definition}

Unramified morphisms are locally separated.
In particular, every Deligne--Mumford stack has locally separated diagonal.

\begin{lemma}\label{L:diagonal-separatedness}
Let $X$ be an algebraic stack. The following are equivalent:
\begin{enumerate}
\item $\Delta_X$ is separated (resp.\ locally separated,
resp.\ quasi-separated).
\item The inertia stack $I_X\to X$ is separated (resp.\ locally separated,
resp.\ quasi-separated).
\item The unit section $X\to I_X$ of the inertia stack is a closed immersion
  (resp.\ an immersion, resp.\ quasi-compact).
\end{enumerate}
\end{lemma}

\begin{lemma}\label{L:diag-separated:descent}
Let $\map{f}{X}{Y}$ be a faithfully flat morphism, locally of finite
presentation, between algebraic stacks.
\begin{enumerate}
\item If $f$ and $\Delta_X$ are separated then so is $\Delta_Y$.
\item If $f$ is representable and $f$ and $\Delta_X$ are locally
separated, then so is $\Delta_Y$.
\end{enumerate}
\end{lemma}
\begin{proof}
This follows from the cartesian diagram
$$\xymatrix{%
X\ar@{(->}[r]\ar[d] & I_X\ar[r] & I_Y\times_Y X\ar[r]\ar[d] & X\ar[d]\\
Y\ar@{(->}[rr] && I_Y\ar[r]\ar@{}[ull]|\square & Y\ar@{}[ul]|\square.\\
}$$
Indeed, it is enough to show that the unit section $Y\inj I_Y$ of the
inertia stack is a closed immersion (resp.\ an immersion). By fppf-descent
(or by noting that $f$ is universally open) it is enough to show that the
morphisms $X\to I_X$ and $I_X\to I_Y\times_Y X$ are proper (resp.\ immersions).
This is the case since the first map is the double diagonal of $X$ and the
second map is a pull-back of the diagonal of $f$.
\end{proof}

\begin{lemma} 
Let $\map{f}{X}{Y}$ be a morphism of stacks.
\begin{enumerate}
\item If $X$ and $\Delta_Y$ are quasi-compact then so is $f$.
\item If $\Delta_X$, and $\Delta_{\Delta_Y}$ are quasi-compact then so is
$\Delta_f$.
\item If $\Delta_{\Delta_X}$ is quasi-compact then so is $\Delta_{\Delta_f}$.
\end{enumerate}
In particular, if $X$ and $Y$ are quasi-separated then so is $f$.
\end{lemma}

\begin{lemma}\label{L:quasi-separatedness}
Let $\map{f}{X}{Y}$ be a surjective morphism of algebraic stacks.
\begin{enumerate}
\item\label{LI:qsep:Y}
If $X$ is quasi-compact then so is $Y$.
\item\label{LI:qsep:DY}
If $f$ and $\Delta_X$ are quasi-compact, then so is $\Delta_Y$.
\item\label{LI:qsep:DDY}
If $f$, $\Delta_f$ and $\Delta_{\Delta_X}$ are quasi-compact, then so is
$\Delta_{\Delta_Y}$.
\end{enumerate}
In particular, if $f$ is quasi-compact and quasi-separated and $X$ is
quasi-separated then $Y$ is quasi-separated.
\end{lemma}
\begin{proof}
\ref{LI:qsep:Y} Choose a commutative diagram
$$\xymatrix{U\ar[r]\ar[d] & V\ar[d]\\
X\ar[r]_f & Y\ar@{}[ul]|\circ}$$
with $U$ and $V$ schemes and such that the vertical morphisms are smooth and
surjective. If $X$ is quasi-compact then we can choose $U$ quasi-compact
and hence also $V$.

\ref{LI:qsep:DY}--\ref{LI:qsep:DDY} The latter statements follows
from~\ref{LI:qsep:Y} and the commutative diagrams
$$\vcenter{\xymatrix{X\ar[r]^-{\Delta_X}\ar[d]_f & X\times X\ar[d]^{f\times f}\\
Y\ar[r]_-{\Delta_Y} & Y\times Y\ar@{}[ul]|\circ}}\quad\text{and}\quad
\vcenter{\xymatrix{X\ar@{(->}[r]^-{\Delta_{\Delta_X}}\ar[d]^f
 & I_X\ar[r]^-g & I_Y\times_Y X\ar[r]\ar[d] & X\ar[d]^f\\
Y\ar@{(->}[rr]_-{\Delta_{\Delta_Y}} && I_Y\ar[r]\ar@{}[ull]|\square
 & Y\ar@{}[ul]|\square}}$$
since $g$ is a pull-back of $\Delta_f$.
\end{proof}

\end{section}


\begin{section}{Algebraicity of points on quasi-separated stacks}
\label{A:alg-points}
Let $X$ be a quasi-separated algebraic stack. In this appendix we show that
every point on $X$ is algebraic, i.e., that for every point $\xi\in |X|$ there
is an algebraic stack $\stG_\xi$ and a monomorphism $\stG_\xi\inj X$ with image
$\xi$ such that $\stG_\xi$ is an fppf-gerbe over the spectrum of a field
$\kappa(\xi)$. In fact, we prove the stronger statement that $\stG_\xi$ is the
generic fiber of a gerbe $Z\to \underline{Z}$ where $Z\inj X$ is an immersion
and $\underline{Z}$ is an integral scheme. It is also enough to assume that $X$
has quasi-compact (but not necessarily quasi-separated) diagonal.

When $X$ is a locally noetherian stack, this result is shown
in~\cite[\S11]{laumon} although the definition of algebraic point is slightly
wrong. The error in~\cite[D\'ef.~11.2]{laumon} is the assertion that if
$\map{x}{\Spec k}{X}$ is \emph{any} representative of $\xi$ and $\Spec k\surj
\stG_x\inj X$ is its epi-mono factorization as fppf stacks, then $\stG_x$ is
independent of the choice of representative $x$. This assertion is not correct
unless restricted to fields $k$ that are finite over the residue field
$\kappa(\xi)$, the reason being that non-finite field extensions are not covering in
the fppf topology. It is possible that the assertion is valid with respect to
the \emph{fpqc topology} but that approach opens up other difficulties as the
epi-mono factorization in the fpqc topology a priori depends on the choice of
universe.

To obtain the algebraicity in full generality we begin with a generic flatness
result due to Raynaud and Gruson.

\begin{theorem}[Generic flatness]\label{T:generic-flatness}
Let $Y$ be an integral scheme. Let $X$ be an algebraic stack and let
$\map{f}{X}{Y}$ be a morphism of finite type. Then there exists an open dense
subscheme $Y_0\subseteq Y$ such that $f|_{Y_0}$ is flat and locally of finite
presentation.
\end{theorem}
\begin{proof}
We can replace $Y$ with an open dense affine subscheme. We can also replace
$X$ with a presentation and assume that $X$ is affine.

Choose a closed
immersion $\injmap{j}{X}{\overline{X}=\A{n}_Y}$. Let $U\subseteq |X|$ be the
locus where $f$ is flat and let $\overline{U}\subseteq
|\overline{X}|$ be the locus where $j_*\sO_X$ is flat over $Y$. Then
$\overline{U}=\overline{X}\setminus j(X\setminus U)$ so that
$U=j^{-1}(\overline{U})$. According to Raynaud and
Gruson~\cite[Thm.~3.4.6]{raynaud-gruson}, the subset $\overline{U}\subseteq
|\overline{X}|$ is open and $(j_*\sO_X)|_{\overline{U}}$ is an
$\sO_{\overline{U}}$-module of finite presentation. Equivalently, we have that
$U\subseteq X$ is open and $f|_U$ is locally of finite presentation.

It remains to find an open dense subset $V\subseteq Y$ such that $f^{-1}(V)$ is
contained in $U$. We let $V=Y\setminus \overline{f(X\setminus U)}$ which
suffices if we can show that $V$ is not the empty set. Since $f$ is
quasi-compact, it follows that $f(X\setminus U)$ is pro-constructible and
hence that $\overline{f(X\setminus U)}$ coincides with the specialization of
$f(X\setminus U)$~\cite[Cor.~7.3.2]{egaI_NE}. Since $f$ is trivially flat at
the generic point of $Y$, it follows that $V$ is non-empty.
\end{proof}

\begin{theorem}[Algebraicity of points]\label{T:algebraicity-of-points}
Let $X$ be an algebraic stack with quasi-compact diagonal. Let $\xi\in |X|$ be
a point. Then there is a quasi-compact immersion $Z\inj X$ such that
\begin{enumerate}
\item $\xi\in |Z|$,
\item\label{TI:alg-pts:flat-inertia}
The inertia stack $I_Z\to Z$ is flat and locally of finite presentation,
\item The stack $Z$ is an fppf-gerbe over an affine scheme
  $\underline{Z}$. The structure morphism $\map{\pi}{Z}{\underline{Z}}$ is
  faithfully flat and locally of finite presentation.
  The scheme $\underline{Z}$ is integral with generic point
  $\underline{\xi}=\pi(\xi)$.
\end{enumerate}
In particular, $\xi\in |X|$ is algebraic with residual gerbe
$\stG_\xi=\pi^{-1}(\underline{\xi})$ and residual field
$\kappa(\underline{\xi})$
and the monomorphism $\stG_\xi\inj X$ is quasi-affine.
\end{theorem}
\begin{proof}
We can replace $X$ with the reduced closed substack $\overline{\{\xi\}}$ so
that $|X|$ is irreducible. To show \ref{TI:alg-pts:flat-inertia}, it is then
enough to show that $I_X\to
X$ is flat and locally of finite presentation over a non-empty quasi-compact
open subset $Z\subseteq X$. Let $\map{p}{U}{X}$ be a smooth presentation with
$U$ a scheme. We can replace $U$ with an affine non-empty open subscheme and
$X$ with its image and assume that $U$ is affine. Let $\map{x}{\Spec k}{X}$ be
a representative of $\xi$. Since $X$ has quasi-compact diagonal, it follows
that $x$ is quasi-compact. Thus $x^{-1}(U)\to U$ is quasi-compact so that
$x^{-1}(U)$ is an algebraic space of finite type over $\Spec k$. Let $W\to
x^{-1}(U)$ be an \etale{} presentation with $W$ an affine scheme of finite type
over $\Spec k$.

As $p$ is open we have that $U$ is the closure of the image of $W$. As $W$ has
a finite number of irreducible components, so has $U$. We can thus replace $U$
by an open non-empty irreducible quasi-compact subscheme and assume that $U$ is
an integral scheme. It now follows from Theorem~\pref{T:generic-flatness} that
$I_X\times_X U\to U$ is flat and locally of finite presentation over an open
dense subscheme $U_0\subseteq U$. We let $Z=p(U_0)$ and
\ref{TI:alg-pts:flat-inertia} follows by flat descent.

Now, as $I_Z\to Z$ is flat and locally of finite presentation, we have that
the fppf-sheafification $\underline{Z}$ of $Z$ is an algebraic space and that
$Z\to \underline{Z}$ is faithfully flat and locally of finite
presentation~\cite[Cor.~10.8]{laumon}. Moreover as the diagonal of $Z$ is
quasi-compact and $\Delta_{Z/\underline{Z}}$ is surjective, it follows that the
diagonal of $\underline{Z}$ is quasi-compact, i.e., that $\underline{Z}$ is a
quasi-separated algebraic space. After replacing $\underline{Z}$ with a dense
open we can thus assume that $\underline{Z}$ is a \emph{scheme}.

Since $Z$ is reduced with generic point $\xi$, we have that $\underline{Z}$ is
reduced with generic point $\underline{\xi}$. We may thus replace
$\underline{Z}$ with an open dense subscheme so that $\underline{Z}$ becomes
affine.
\end{proof}
\end{section}


\begin{section}{Two lemmas on algebraic spaces}\label{A:lemmas}
In this appendix we state two lemmas on algebraic spaces that likely are
well-known to experts.

\begin{lemma}\label{L:finite-flat-pres-of-alg-space:aff/qaff}
Let $X$ be an algebraic space and let $\map{p}{X'}{X}$ be a finite flat
presentation by an affine (resp.\ quasi-affine) scheme $X'$. Then $X$ is an
affine (resp.\ quasi-affine) scheme.
\end{lemma}
\begin{proof}
Every fiber of $p$ has an affine open neighborhood
by~\cite[Cor.~4.5.4]{egaII}. Hence $X$ is an affine scheme (resp.\ a
scheme)~\cite[Exp.~V, Thm.~4.1]{sga3}. If $X'$ is quasi-affine then so is
$X$ by~\cite[Cor.~6.6.3]{egaII}.
\end{proof}

\begin{lemma}\label{L:qf+loc-sep-over-strictly-local}
Let $S=\Spec A$ be strictly local, i.e., let $A$ be a strictly henselian local
ring. Let $X$ be an algebraic space and let $X\to S$ be locally quasi-finite
and \emph{locally separated}. Let $\map{\overline{x}}{\Spec k}{X}$ be a
geometric point over the closed point $s\in S$ and let $\sO_{X,\overline{x}}$
denote the strictly local ring. Then $\map{g}{\Spec(\sO_{X,\overline{x}})}{X}$
is an open immersion and $\Spec(\sO_{X,\overline{x}})\to S$ is finite.
\end{lemma}
\begin{proof}
The lemma is well-known for schemes~\cite[Thm.~18.5.11]{egaIV}.
Let $U\to X$ be an \etale{}
presentation with $U$ a scheme. Then $\overline{x}$ lifts to $U$ so that
$Z=\Spec(\sO_{X,\overline{x}})$ is an open subscheme of $U$ and $Z\to S$ is
finite. It follows that $\map{g}{Z}{X}$ is \etale{}. The scheme $Z\times_S Z$ is
local. By assumption, the morphism $Z\times_X Z\to Z\times_S Z$ is an immersion
and as the closed point of $Z\times_S Z$ lies in $Z\times_X Z$ it follows that
$Z\times_X Z\to Z\times_S Z$ is a closed immersion. In particular, $Z\times_X Z$
is finite over $Z$ so that $Z\to g(Z)$ is finite and \etale{}. Since $g$ has
rank $1$ at $x$ it follows that $Z\to g(Z)$
is an isomorphism so that $g$ is an open immersion.
\end{proof}

\end{section}


\begin{section}{$2$-sheaves on the category of algebraic stacks}\label{A:2-sheaves}
In this appendix we define $2$-sheaves on the $2$-category of algebraic stacks
with the \etale{} topology. A $2$-presheaf (resp.\ a $2$-sheaf) is a
generalization of the notion of a fibered category (resp.\ a stack) that
allows the base category (resp.\ site) to be a $2$-category. We have chosen to
describe $2$-presheaves in terms of $2$-functors and not in terms of fibered
$2$-categories as this appears to be the simplest description. Similarly, one
usually describe $1$-sheaves as ordinary functors and not as fibered categories
with fibers equivalent to discrete categories. There are essentially three
different ways to describe the sheaf condition: with sieves, (semi\nobreakdash-)simplicial
objects or classical descent data. Our presentation takes the classical
approach.

The general theory of $2$-sheaves has been developed by R.\ Street in two
papers. The first paper~\cite{street_two-dim-sheaf-theory} treats the, from our
perspective, less interesting case where \emph{all} notions are strict. The
second paper~\cite{street_char-bicat-stacks} briefly treats the non-strict case
(which generalizes fibered categories and stacks) but the proofs have to be
copied and modified from the first paper. We have therefore decided to make the
following presentation independent of these two papers. To further simplify the
discussion, the results are stated for \emph{strict} $2$-presheaves although
the results remain valid for arbitrary $2$-presheaves. This latter notion of
strictness should not be confused with the all-encompassing strictness imposed
in~\cite{street_two-dim-sheaf-theory}. Moreover, every $2$-presheaf is
equivalent (but not isomorphic) to a strict $2$-presheaf by the bicategorical
Yoneda lemma so we do not lose anything by limiting ourselves to the
$2$-category of strict $2$-presheaves.

\begin{definition}
A $2$-category is a category $\catC$ enriched in categories, i.e., for every
pair of objects $(X,Y)$ in $\catC$ we have a category
$\catHom_{\catC}(X,Y)$. The objects (resp.\ arrows) of $\catHom_{\catC}(X,Y)$
are called $1$-morphisms (resp.\ $2$-morphisms). We say that $\catC$ is a
$(2,1)$-category if every $2$-morphism is invertible, i.e., if
$\catHom_{\catC}(X,Y)$ is a groupoid for every $(X,Y)$.
\end{definition}

The standard example of a $2$-category is the $2$-category $\Cat$ of
categories, functors and natural transformations. Similarly, the standard
example of a $(2,1)$-category is the full subcategory $\Grpd\subseteq \Cat$
of groupoids. The other important example of a $(2,1)$-category is
the $(2,1)$-category of algebraic stacks $\Stack$. All these $2$-categories have
$2$-fiber products.

\begin{definition}
A \emph{strict $2$-functor} $\map{F}{\catC}{\catD}$ between $2$-categories
is
\begin{enumerate}
\item a map $\map{F}{\ob \catC}{\ob \catD}$,
\item for every pair of objects $X,Y\in \ob \catC$, a functor
$$\map{F(X,Y)}{\catHom_{\catC}(X,Y)}{\catHom_{\catD}(F(X),F(Y))},$$
\end{enumerate}
such that
\renewcommand{\theenumi}{{\upshape{(\alph{enumi})}}}
\begin{enumerate}
\item $F(\id{X})=\id{F(X)}$ for every $x\in \ob\catC$,\label{DI:s2f-id}
\item for every $X,Y,Z\in \catC$, the diagram
$$\xymatrix{ \catHom_{\catC}(Y,Z)\times \catHom_{\catC}(X,Y)\ar[r]^-{\circ}
\ar[d]_{F(Y,Z)\times F(X,Y)}
 & \catHom_{\catC}(X,Z)\ar[d]^{F(X,Z)}\\
\catHom_{\catD}(FY,FZ)\times \catHom_{\catD}(FX,FY)\ar[r]^-{\circ} & \catHom_{\catD}(FX,FZ)}$$
is strictly commutative.\label{DI:s2f-comp}
\end{enumerate}
\end{definition}

A $2$-functor is defined similarly but instead of requiring that the functor
respects identities and composition as in~\ref{DI:s2f-id}
and~\ref{DI:s2f-comp}, the data for a $2$-functor include
$2$-isomorphisms $\id{F(X)}\Rightarrow F(\id{X})$ and natural isomorphisms
between the two functors $\catHom(Y,Z)\times \catHom(X,Y)\to
\catHom(FX,FZ)$. These isomorphisms are then required to satisfy natural
coherence
conditions. In some literature $2$-functors are called \emph{pseudofunctors}
and strict $2$-functors are simply called $2$-functors.

\begin{definition}
Let $\catC$ be a $2$-category. A (strict) $2$-presheaf on $\catC$ is a (strict)
$2$-functor $\map{\tfF}{\catC^{\op}}{\Cat}$.
\end{definition}

The general definition of topologies and $2$-sheaves on $2$-categories can be
found in~\cite{street_char-bicat-stacks}. In the remainder of
this appendix we will give a concrete description of strict $2$-sheaves on
the $(2,1)$-category $\Stack$ of algebraic stacks with the \etale{} topology.
Fix a strict $2$-presheaf
$\map{\tfF}{\Stack^{\op}}{\Cat}$.

\begin{definition}
A family of morphisms $(\map{p_\alpha}{X_\alpha}{X})$ in $\Stack$ is
\emph{covering in the \etale{} topology} if the $p_\alpha$'s are smooth
and $\coprod_\alpha X_\alpha\to X$ is surjective.
\end{definition}

Note that the $p_\alpha$'s need not be representable.

\begin{xpar}\label{X:X,X',X''}
Let $(\map{p_\alpha}{X_\alpha}{X})_\alpha$ be a family of morphisms in
$\Stack$. Let $X_{\alpha\beta}=X_\alpha\times_X X_\beta$ denote a $2$-fiber product
with projections $\map{\pi_1}{X_{\alpha\beta}}{X_\alpha}$ and
$\map{\pi_2}{X_{\alpha\beta}}{X_\beta}$ and a $2$-isomorphism
$p_\alpha\circ\pi_1\Rightarrow p_\beta\circ\pi_2$. This
induces a natural isomorphism of functors
$$\map{\pi_1^*p_\alpha^* = (p_\alpha\circ \pi_1)^*
\iso (p_\beta\circ \pi_2)^* = \pi_2^*p_\beta^*}
{\tfF(X)}{\tfF(X_{\alpha\beta}).}$$
In particular, if $\map{f}{\sF}{\sG}$ is a morphism in $\tfF(X)$, then we
obtain a commutative diagram
$$\xymatrix{\pi_1^*p_\alpha^*\sF\ar[r]^{\pi_1^*p_\alpha^*f}\ar[d]^{\iso}
& \pi_1^*p_\alpha^*\sG\ar[d]^{\iso}\\
\pi_2^*p_\beta^*\sF\ar[r]^{\pi_2^*p_\beta^*f} & \pi_2^*p_\beta^*\sG.}$$
\end{xpar}

\begin{definition}
A family of morphisms $(\map{p_\alpha}{X_\alpha}{X})_\alpha$ in $\Stack$ is a
\emph{family of descent} for $\tfF$ if for every $\sF,\sG\in \tfF(X)$ the
sequence
$$\xymatrix{\Hom_{\tfF(X)}(\sF,\sG)\ar[r]^-{(p_\alpha^*)}
& \prod_\alpha \Hom_{\tfF(X_\alpha)}(\sF_\alpha,\sG_\alpha)
              \ar@<.5ex>[r]^-{\pi_1^*}\ar@<-.5ex>[r]_-{\pi_2^*}
& \prod_{\alpha\beta} \Hom_{\tfF(X_{\alpha\beta})}(\sF_{\alpha\beta},\sG_{\alpha\beta})}$$
is exact where $\sF_\alpha=p_\alpha^*\sF$,
$\sF_{\alpha\beta}=\pi_1^*p_\alpha^*\sF$ and we implicitly have used the
canonical isomorphism $\pi_1^*p_\alpha^*\sF\to \pi_2^*p_\beta^*\sF$ (and
similarly for $\sG$).
\end{definition}

\begin{xpar}[Cocycle condition]
Let $(\map{p_\alpha}{X_\alpha}{X})_\alpha$ and $\pi_1,\pi_2$ be as
in~\pref{X:X,X',X''}. Given an object $\sF\in \tfF(X)$ we saw that we obtained
a canonical isomorphism
$\isomap{\psi_{\alpha\beta}}{\pi_1^*\sF_\alpha}{\pi_2^*\sF_\beta}$ in
$\tfF(X_{\alpha\beta})$ where $\sF_\alpha=p_\alpha^*\sF$. The isomorphism
$\psi_{\alpha\beta}$ satisfies the
cocycle condition, i.e., the following diagram in $\tfF(X_{\alpha\beta\gamma})$
$$\xymatrix@C=1mm@R=8.6mm{%
& \pi_{12}^{*}\pi_2^{*}\sF_\beta\ar[rrrr]^{\can}_{\iso}
&&&& \pi_{23}^{*}\pi_1^{*}\sF_\beta\ar[dr]^{\pi_{23}^*(\psi_{\beta\gamma})}
& \\
\pi_{12}^{*}\pi_1^{*}\sF_\alpha\ar[ur]^{\pi_{12}^*(\psi_{\alpha\beta})}
&&&\circ&&& \pi_{23}^{*}\pi_2^{*}\sF_\gamma\ar[dl]_{\can}^{\iso} \\
& \pi_{31}^{*}\pi_2^{*}\sF_\alpha\ar[ul]_{\can}^{\iso}
&&&& \pi_{31}^{*}\pi_1^{*}\sF_\gamma\ar[llll]_{\pi_{31}^*(\psi_{\gamma\alpha})} &
}$$
commutes. Here $X_{\alpha\beta\gamma}=X_\alpha\times_X X_\beta\times_X
X_\gamma$ is a $2$-fiber product with projections
$\map{\pi_{ij}}{X_{\alpha_1\alpha_2\alpha_3}}{X_{\alpha_i\alpha_j}}$
onto the $i$\textsuperscript{th} and $j$\textsuperscript{th} factors
and the maps denoted with ``can'' are canonical isomorphisms.
\end{xpar}

\begin{definition}
Let $(\sF_\alpha)_\alpha\in \prod_\alpha\tfF(X_\alpha)$. A \emph{descent datum}
for $(\sF_\alpha)_\alpha$ is a collection of isomorphisms
$\map{\psi_{\alpha\beta}}{\pi_1^*\sF_\alpha}{\pi_2^*\sF_\beta}$ in
$\tfF(X_{\alpha\beta})$ satisfying the cocycle condition.
\end{definition}

\begin{definition}
We let $\tfF\bigl((p_\alpha)_\alpha\bigr)=\tfF\bigl((X_\alpha\to
X)_\alpha\bigr)$ be the category with
\begin{itemize}
\item objects: pairs $\bigl((\sF_\alpha),(\psi_{\alpha\beta})\bigr)$ of an
object $(\sF_\alpha) \in \prod_\alpha \tfF(X_\alpha)$ equipped with a descent
datum $(\psi_{\alpha\beta})$.
\item morphisms $\bigl((\sF_\alpha),(\psi_{\alpha\beta})\bigr)\to
\bigl((\sG_\alpha),(\theta_{\alpha\beta})\bigr)$: a morphism
$\map{(f_\alpha)}{(\sF_\alpha)}{(\sG_\alpha)}$ in $\prod_\alpha \tfF(X_\alpha)$
such that
$$\xymatrix{\pi_1^*\sF_\alpha\ar[r]^{\pi_1^*f_\alpha}\ar[d]^{\psi_{\alpha\beta}}
& \pi_1^*\sG_\alpha\ar[d]^{\theta_{\alpha\beta}}\\
\pi_2^*\sF_\beta\ar[r]^{\pi_2^*f_\beta} & \pi_2^*\sG_\beta}$$
commutes for all pairs $(\alpha,\beta)$.
\end{itemize}
\end{definition}

There is a natural functor $\map{(p_\alpha)^{*}_{D}}{\tfF(X)}{\tfF\bigl((p_\alpha)_\alpha\bigr)}$
taking an object $\sF\in \tfF(X)$ onto $(p_\alpha^*\sF)$ equipped with the
induced descent datum and taking a morphism $\sF\to \sG$ onto
$(p_\alpha^*\sF\to p_\alpha^*\sG)$. The functor $(p_\alpha)^{*}_{D}$ is fully
faithful if and only if $(p_\alpha)$ is a family of descent for $\tfF$.

\begin{definition}
A family of descent $(p_\alpha)$ is a family of \emph{effective descent} for
$\tfF$ if $\map{(p_\alpha)^{*}_D}{\tfF(X)}{\tfF\bigl((p_\alpha)\bigr)}$ is an
equivalence of categories.
\end{definition}

\begin{definition}
A $2$-presheaf $\map{\tfF}{\Stack^{\op}}{\Cat}$ is a \emph{$2$-sheaf in the
\etale{}
topology} if every covering family $(X_\alpha\to X)_\alpha$ in the \etale{}
topology
is of effective descent for $\tfF$.
\end{definition}

\begin{definition}
We let $\Stack'$ denote the subcategory of $\Stack$ with all objects but
with only smooth $1$-morphisms. For an algebraic stack
$X$, we let $\Stack_{/X}$ denote the $2$-category of morphisms $Z\to X$. We let
$\Stack_{\metale/X}\subset \Stack_{\smooth/X}\subset \Stack_{/X}$ denote the
full $2$-subcategories of \etale{} and smooth morphisms and we let
$\Stack_{\repr/X}\subset \Stack_{/X}$ denote the full $1$-category of
representable morphisms. We say that a family of morphisms in $\Stack_{/X}$ is
covering if its image in $\Stack$ is covering. We say that a $2$-presheaf on
any of these categories is a $2$-sheaf if every covering family is of effective
descent.
\end{definition}

If $\tfF$ is a $2$-sheaf on $\Stack$ (resp.\ $\Stack'$) then the restricted
$2$-presheaf on $\Stack_{/X}$ (resp.\ $\Stack_{\smooth/X}$) is a $2$-sheaf for
any stack $X$. In particular, the restriction to $\Stack_{\metale/X}$ is a
$2$-sheaf.

By the
comparison lemma for $2$-sheaves~\cite[Thm.~3.8]{street_two-dim-sheaf-theory}
restriction along $\Stack_{\repr/X}\subset \Stack_{/X}$
induces a $2$-equivalence between the $2$-category of $2$-sheaves on
$\Stack_{/X}$ and the $2$-category of $2$-sheaves on $\Stack_{\repr/X}$,
or equivalently, the $2$-category of stacks on $\Stack_{\repr/X}$. The
following result
is essentially a reformulation of the comparison lemma (also
see~\cite[Prop.~10.10]{giraud_descente}).

\begin{proposition}\label{P:descent-representable}
Let $\map{\tfF}{\Stack^{\op}}{\Cat}$ be a strict $2$-presheaf. Then $\tfF$ is
a $2$-sheaf if and only if the following two conditions hold.
\begin{enumerate}
\item\label{PI:descent-repr:pres}
For every algebraic stack $X$ and every surjective smooth morphism
$\map{p}{U}{X}$ such that $U$ is an \emph{algebraic space}, we have that
$p$ is of effective descent.
\item\label{PI:descent-repr:coprod}
For every family of algebraic spaces $(X_\alpha)$ the natural functor
$$\tfF\left(\coprod_\alpha X_\alpha\right)\to \prod_\alpha \tfF(X_\alpha)$$
is an equivalence of categories.
\end{enumerate}
\end{proposition}
\begin{proof}
The functor in \ref{PI:descent-repr:coprod} is an equivalence if and only if
for every algebraic space
$X=\coprod X_\alpha$, the family $(\injmap{j_\alpha}{X_\alpha}{X})$ is of
effective descent, cf.\ \cite[Prop.~9.24]{giraud_descente}. Thus, the two
conditions are necessary. Moreover, if \ref{PI:descent-repr:coprod} holds, then
for every covering
family $(\map{p_\alpha}{X_\alpha}{X})$ such that the $X_\alpha$'s are algebraic
spaces, the natural functor
$$\tfF\left(\map{\coprod_\alpha p_\alpha}{\coprod_\alpha X_\alpha}{X}\right)\to \tfF\bigl((\map{p_\alpha}{X_\alpha}{X})\bigr)$$
is an equivalence.

To show that the conditions are sufficient, assume that
\ref{PI:descent-repr:pres} and~\ref{PI:descent-repr:coprod} hold and
let $(\map{p_\alpha}{X_\alpha}{X})$ be a covering family. For every $\alpha$
choose a smooth presentation $\map{q_\alpha}{U_\alpha}{X_\alpha}$ so that we
obtain $2$-commutative diagrams
%
$$\xymatrix{%
U_{\alpha\beta}:=U_\alpha\times_X U_\beta\ar[d]^{q_{\alpha\beta}}\ar[r]^-{\pi_1}
 & U_\alpha\ar[d]^{q_{\alpha}}\\
X_{\alpha\beta}:=X_\alpha\times_X X_\beta\ar[r]^-{\pi_1} & X_\alpha\ar[r]^{p_\alpha} & X}$$
%
and similarly for $\pi_2$. With the usual choice of the $2$-fiber product
$X_{\alpha\beta}$, we can even assume that the diagram is strictly
commutative.
There is a natural functor
$$\map{Q}{\tfF\bigl((p_\alpha)_\alpha\bigr)}{\tfF\bigl((p_\alpha\circ q_\alpha)_\alpha\bigr)}$$
taking an object $\bigl((\sF_\alpha),(\psi_{\alpha\beta})\bigr)$ to
$\bigl((q_\alpha^*\sF_\alpha),(q_{\alpha\beta}^*\psi_{\alpha\beta})\bigr)$
so that ${(p_\alpha\circ q_\alpha)^*_D}=Q\circ (p_\alpha)^*_D$.

Since $q_\alpha$ and $q_{\alpha\beta}$ are of descent it follows that $Q$ is
fully faithful. Indeed, that $Q$ is faithful is immediate from the faithfulness
of $q_\alpha^*$. To see that $Q$ is full, let
$\bigl((\sF_\alpha),(\psi_{\alpha\beta})\bigr)$ and
$\bigl((\sG_\alpha),(\theta_{\alpha\beta})\bigr)$ be objects of
$\tfF(p_\alpha)$ and let
$\map{(g_\alpha)}{(q_\alpha^*\sF_\alpha)}{(q_\alpha^*\sG_\alpha)}$ be a
morphism in $\tfF(p_\alpha\circ q_\alpha)$. Since $(q_\alpha)^*_D$ is full
$g_\alpha$ descends to a map $\map{(f_\alpha)}{(\sF_\alpha)}{(\sG_\alpha)}$.
That this map is compatible with the descent data $\psi$ and $\theta$ follows
from the faithfulness of $q_{\alpha\beta}^*$.

As $(p_\alpha\circ q_\alpha)$ is of effective descent, it follows that
$Q$ is essentially surjective and hence an equivalence of categories. It
follows that $(p_\alpha)$ is a family of effective descent and that
$\tfF$ is a $2$-sheaf.
\end{proof}

More generally, the proposition holds for non-strict $2$-presheaves. In this
case, we only have a natural isomorphism
$(p_\alpha\circ q_\alpha)^*_D\iso Q\circ (p_\alpha)^*_D$ in the proof.

\end{section}


\begin{section}{Examples of $2$-sheaves}\label{A:ex-2-sheaves}
In this appendix we show that the $2$-presheaf $\catHom(-,Y)$ of morphisms to a
given stack $Y$ and the $2$-presheaf $\tfQCoh(-)$ of quasi-coherent sheaves are
$2$-sheaves. For simplicity, we only treat the restriction of $\tfQCoh$ to the
full subcategory $\Stack'\subset \Stack$ so that $\tfQCoh$ is a
\emph{strict} $2$-sheaf.

Let $Y$ be an algebraic stack. There is a strict $2$-functor
$$\map{\catHom(-,Y)}{\Stack^{\op}}{\Grpd}$$
that takes an algebraic stack $X$ to the groupoid $\catHom(X,Y)$.

\begin{theorem}\label{T:Hom-is-2-sheaf}
The strict $2$-presheaf $\map{\catHom(-,Y)}{\Stack^{\op}}{\Grpd}$ is a
$2$-sheaf.
\end{theorem}
\begin{proof}
Let us verify the conditions of
Proposition~\pref{P:descent-representable}. Part~\ref{PI:descent-repr:pres} is
the description of
$\catHom(X,Y)$ given in~\cite[pf.\ of Prop.~4.18]{laumon} and
part~\ref{PI:descent-repr:coprod} is the
definition of the coproduct in the $2$-category of stacks.
\end{proof}

There is a strict $2$-functor
$$\map{\Mod_{\cart}(\sO_{-})}{\Stack'^{\op}}{\Cat}$$
taking an algebraic stack $X$ to the category $\Mod_{\cart}(\sO_X)$ of
cartesian lisse-\etale{} $\sO_X$-modules. For a
smooth morphism $\map{f}{X'}{X}$ the pull-back
$\map{f^*}{\Mod_{\cart}(\sO_X)}{\Mod_{\cart}(\sO_{X'})}$ is defined by
restriction along the functor $\tLiset(X')\to \tLiset(X)$. This defines $f^*$
uniquely (not merely up to isomorphism) so that $\Mod_{\cart}(\sO_{-})$ is
indeed a strict $2$-presheaf. We let
$$\map{\tfQCoh(-)}{\Stack'^{\op}}{\Cat}$$
be the sub-$2$-presheaf of quasi-coherent modules.

\begin{theorem}
The strict $2$-presheaf $\map{\tfQCoh(-)}{\Stack'^{\op}}{\Cat}$ is a $2$-sheaf.
\end{theorem}
\begin{proof}
We will verify the conditions of Proposition~\pref{P:descent-representable} for
$\tfQCoh$. Condition~\ref{PI:descent-repr:pres} for the $2$-presheaf
$\Mod_{\cart}(\sO_{-})$
is~\cite[Lem.~4.5]{olsson_sheaves-on-artin-stacks}
and that condition~\ref{PI:descent-repr:coprod} holds is obvious.

It remains to verify that if $X$ is an algebraic stack, if $\sF$ is a cartesian
$\sO_X$-module, and if $\map{p}{U}{X}$ is a smooth presentation, then $\sF$ is
a quasi-coherent $\sO_X$-module if and only if $p^{*}\sF$ is a quasi-coherent
$\sO_U$-module. This is~\cite[Prop.~13.2.1]{laumon}.
\end{proof}

It is also not difficult to show that the $2$-presheaves of:
big-\etale{} sheaves, lisse-\etale{} sheaves,
cartesian lisse-\etale{} sheaves, constructible sheaves,
$\sA$-modules, cartesian $\sA$-modules; are all $2$-sheaves. Here $\sA$ denotes
a flat lisse-\etale{} sheaf of rings. Another important $2$-sheaf is the
$2$-presheaf $\tfStackRep$ that takes a stack $X$ to the $1$-category
$\Stack_{\repr/X}$
of representable morphisms $Z\to X$. This is a subsheaf of the $2$-sheaf of
big-\etale{} sheaves. That $\tfStackRep$ is a $2$-sheaf follows
from~\cite[Cor.~10.5]{laumon}.
\end{section}


\bibliography{etnbhds}
\bibliographystyle{dary}

\end{document}